\documentclass[10pt]{article}
\usepackage{latexsym}
\usepackage{amssymb}
\usepackage{amsmath}
\usepackage{amsthm}
\usepackage{thmtools}
\usepackage{cite}
\usepackage[all]{xy}
\usepackage{framed}
\usepackage{hyperref}
\usepackage{graphicx}
\usepackage{parskip}
\usepackage{hyperref}

\begingroup
    \makeatletter
    \@for\theoremstyle:=definition,remark,plain\do{%
        \expandafter\g@addto@macro\csname th@\theoremstyle\endcsname{%
            \addtolength\thm@preskip\parskip
            }%
        }
\endgroup

\makeatletter
\xydef@\txt@ii#1{\vbox{\vspace*{-5pt}%
 \let\\=\cr
 \tabskip=\z@skip \halign{\relax\hfil\txtline@@{##}\hfil\cr\leavevmode#1\crcr}}}
\makeatother

\theoremstyle{definition}
\newtheorem{thm}{Theorem}[section]
\newtheorem{lem}[thm]{Lemma}
\newtheorem{cor}[thm]{Corollary}
\newtheorem{defn}[thm]{Definition}
\newtheorem{propn}[thm]{Proposition}
\newtheorem*{thm*}{Theorem}
\newtheorem*{notn}{Notation}

\newtheorem*{nts}{Note to self}

\theoremstyle{remark}
\newtheorem*{rk}{Remark}
\newtheorem{ex}[thm]{Example}

\newtheoremstyle{custthm}{\parskip}{}{\normalfont}{}{\bfseries}{.}{ }{\thmname{#1} \thmnote{#3}}
\theoremstyle{custthm}
\newtheorem*{letterthm}{Theorem}

\newcommand{\tensor}[1]{\underset{#1}{\otimes}}

\newcommand{\gr}{\mathrm{gr}}

\newcommand{\Sat}{\mathrm{Sat}\,}

\newcommand{\inn}{\mathrm{inn}}
\newcommand{\Aut}{\mathrm{Aut}}

\newcommand{\fn}{\mathbf{FN}_p}
\newcommand{\Rinfty}{\mathbb{R}\cup\{\infty\}}

\newcommand{\cpi}{\mathrm{cpi}^{\overline{k\Delta^+}}}
\newcommand{\cpisum}[2]{#1|^{#2}}
\newcommand{\cp}[2]{\underset{\langle #1,#2 \rangle}{*}}
\newcommand{\Xinn}[3]{\mathrm{Xinn}_{#1}({#2}; {#3})}
\newcommand{\cpnormal}[1]{\textsf{N}${}_{#1}$}
\newcommand{\taunormal}[1]{\textsf{N}${}'_{#1}$}
\newcommand{\tausub}[1]{\textsf{P}${}_{#1}$}
\newcommand{\tauinv}[1]{\textsf{Q}${}_{#1}$}
\newcommand{\omegaprop}[1]{\textsf{A}${}_{#1}$}

\newenvironment{rcases}
  {\left.\begin{aligned}}
  {\end{aligned}\right\rbrace}

\newcommand\blfootnote[1]{%
  \begingroup
  \renewcommand\thefootnote{}\footnote{#1}%
  \addtocounter{footnote}{-1}%
  \endgroup
}

{%
\begin{framed}\begin{nts}%
}%
{%
\end{nts}\end{framed}%
}

\begin{document}

\numberwithin{equation}{section}
\binoppenalty=\maxdimen
\relpenalty=\maxdimen

\title{Extensions of almost faithful prime ideals in virtually nilpotent mod-$p$ Iwasawa algebras}
\author{William Woods}
\date{\today}
\maketitle
\begin{abstract}
Let $G$ be a nilpotent-by-finite compact $p$-adic analytic group for some $p>2$, and $H = \fn(G)$ its finite-by-(nilpotent $p$-valuable) radical. Fix a finite field $k$ of characteristic $p$, and write $kG$ for the completed group ring of $G$ over $k$. We show that almost faithful $G$-stable prime ideals $P$ of $kH$ extend to prime ideals $PkG$ of $kG$.
\end{abstract}

\newpage
\tableofcontents

%
%
%
%
%
%

\newpage
\section*{Introduction}

Let $G$ be a nilpotent-by-finite compact $p$-adic analytic group and $k$ a finite field of characteristic $p$.

Recall the characteristic open subgroup $H = \fn(G)$, the \emph{finite-by-(nilpotent $p$-valuable) radical} of $G$, defined in \cite[Theorem C]{woods-struct-of-G}. This plays an important role in the structure of the group $G$: for instance, see the structure theorem \cite[Theorem D]{woods-struct-of-G}.

In this paper, we demonstrate a connection between certain prime ideals of $kH$ and those of $kG$. The main result of this paper is:

\begin{letterthm}[A]
Fix some prime $p>2$. Let $G$ be a nilpotent-by-finite compact $p$-adic analytic group, $H = \fn(G)$, and $k$ a finite field of characteristic $p$. Let $P$ be an almost faithful, $G$-stable prime ideal of $kH$. Then $PkG$ is a prime ideal of $kG$.\qed
\end{letterthm}

The proof (given in Propositions \ref{propn: main prime extension theorem, Delta+ = 1} and \ref{propn: main prime extension theorem, Delta+ not equal to 1}) comprises several technical elements, which we outline below.

First, let $G$ be a nilpotent-by-finite compact $p$-adic analytic group with finite radical $\Delta^+ = 1$ \cite[Definition 1.2]{woods-struct-of-G}, and let $H = \fn(G)$. Note that $H$ is $p$-valuable \cite[III, 2.1.2]{lazard}, and that $G$ acts on the set of $p$-valuations of $H$ as follows: if $\alpha$ is a $p$-valuation on $H$ and $g\in G$, then we may define a new $p$-valuation $g\cdot\alpha$ on $H$ by
$$g\cdot\alpha(x) = \alpha(g^{-1}xg).$$
(In fact, we do this in a slightly more general case, but the details are identical. See Lemma \ref{lem: group action on p-valuations} for the setup.)

Recall the definition of an \emph{isolated} orbital (closed) subgroup $L$ of $H$ from \cite[Definition 1.4]{woods-struct-of-G}, and that normal subgroups are automatically orbital. We show in Definition \ref{defn: quotient p-val} that, if $\omega$ is a $p$-valuation on $H$ and $L$ is a closed isolated normal subgroup of $H$, then $\omega$ induces a \emph{quotient} $p$-valuation $\Omega$ on $H/L$. We also define the \emph{$(t,p)$-filtration} (actually a $p$-valuation) on a free abelian pro-$p$ group $A$ of finite rank in Definition \ref{defn: (t,p)-filtration}: this is a particularly ``uniform" $p$-valuation on $A$, analogous to the $p$-adic valuation $v_p$ on $\mathbb{Z}_p$.

It is now easy to show the following.

\begin{letterthm}[B]
With the above notation: let $L$ be a proper closed isolated normal subgroup of $H$ containing the commutator subgroup $[H,H]$. Then there exists a $p$-valuation $\omega$ on $H$ with the following properties:
\begin{itemize}
\item[(i)] $\omega$ is $G$-invariant,
\item[(ii)] there exists a real number $t > (p-1)^{-1}$ such that $\omega|_L > t$, and the quotient $p$-valuation induced by $\omega$ on $G/L$ is the $(t,p)$-filtration.\qed
\end{itemize}
\end{letterthm}

Continue to take $G$ to be a nilpotent-by-finite compact $p$-adic analytic group with $\Delta^+ = 1$, and $H = \fn(G)$. Let $p$ be a prime, $k$ a field of characteristic $p$, and $P$ a faithful prime ideal of $kH$. It is shown in \cite[Theorem 8.4]{ardakovInv} that $P = \mathfrak{p}kH$ for some prime ideal $\mathfrak{p}$ of $kZ$, where $Z$ is the centre of $H$; and, furthermore, in \cite[proof of Theorem 8.6]{ardakovInv}, that there exist an integer $e$ and a ring filtration $f$ on $kH/P$ such that
\begin{equation}\label{eqn: ardakov gr}
\gr_f(kH/P) \cong (\gr_v(kZ/\mathfrak{p})) [Y_1, \dots, Y_e]\tag{$\dagger$}
\end{equation}
where $v = f|_{kZ/\mathfrak{p}}$ is a valuation, and $\gr_f(kZ/\mathfrak{p})$ is a commutative domain. The valuation $f$ is partly constructed using the $p$-valuation $\omega$ on the group $H$.

Our next theorem is an extension of this result. Write $f_1$ for the ring filtration constructed above. Suppose now that $P$ is $G$-stable so that we may consider the ring $kG/PkG$, and fix a crossed product decomposition $(kH/P) * F$ of this ring, where $F = G/H$.

Write $Q'$ for the classical ring of quotients of $kZ/\mathfrak{p}$. The filtration $f_1$ restricts to a valuation on $kZ/\mathfrak{p}$, which extends naturally to a valuation on $Q'$, which we will call $v_1$. $F$ acts on the set of valuations of $Q'$, and $v_1$ has some orbit $\{v_1, \dots, v_s\}$.

Write $Q$ for a certain partial ring of quotients of $kH/P$ containing $Q'$. We may naturally form $Q*F$ as an overring of $(kH/P) * F$.

\begin{letterthm}[C]
In the above notation: there exists a filtration $\hat{f}$ on $Q*F$ such that
\begin{itemize}
\item[(i)] $\gr_{\hat{f}}(Q*F) \cong \gr_{\hat{f}}(Q)*F$, where the right-hand side is some crossed product,
\item[(ii)] $\gr_{\hat{f}}(Q) \cong \bigoplus_{i=1}^s \gr_{f_i}(Q)$,
\item[(iii)] $\gr_{f_i}(Q) \cong (\gr_{v_i} Q')[Y_1, \dots, Y_e]$ for all $1\leq i\leq s$,
\end{itemize}
where $s$ and $e$ are determined as in (\ref{eqn: ardakov gr}), and the action of $F$ in the crossed product of (i) permutes the $s$ summands in the decomposition of (ii) transitively by conjugation.\qed
\end{letterthm}

We combine Theorems B and C as follows.

Theorem C, of course, only invokes (\ref{eqn: ardakov gr}) in the case when $Q\neq Q'$, so we suppose that we are in this case, which occurs precisely when $H$ is non-abelian. Take $L$ to be the smallest closed isolated normal subgroup of $H$ containing both the isolated derived subgroup $H'$ \cite[Theorem B]{woods-struct-of-G} and the centre $Z$ of $H$. Now $L$ is a proper subgroup by Lemma \ref{lem: H'Z is not all of H}, and we will choose $\omega$ for $L$ as in Theorem B. We may arrange it in (\ref{eqn: ardakov gr}) and Theorem C so that, for some $l\leq e$, the elements $Y_1, \dots, Y_l$ correspond to a $\mathbb{Z}_p$-module basis $x_1, \dots, x_l$ for $H/L$; and here the value of the filtration $\hat{f}$ can be understood in terms of the $p$-valuation $\omega$. We show that:

\begin{letterthm}[D]
Take an automorphism $\sigma$ of $H$. Suppose that the induced automorphism on $\gr_{\hat{f}}(Q*F)$ fixes each of the valuations $v_1, \dots, v_s$ and fixes each of the elements $Y_1, \dots, Y_l$. Then the induced automorphism on $H/L$ (which can be seen as a matrix $M_\sigma \in GL_l(\mathbb{Z}_p)$) lies in the first congruence subgroup of $GL_l(\mathbb{Z}_p)$, i.e. it takes the form $M_\sigma \in 1 + pX$ for some $X \in M_l(\mathbb{Z}_p)$. In particular, when $p>2$, $\sigma$ has finite order if and only if it is the identity automorphism.\qed
\end{letterthm}

A special case of Theorem A, in which $\Delta^+ = 1$, is now deduced from Theorem D via a long but elementary argument about X-inner automorphisms: see Definition \ref{defn: X-inner stuff} and Proposition \ref{propn: main prime extension theorem, Delta+ = 1} for details.

The case when $\Delta^+ \neq 1$ now follows as a consequence of the ``untwisting" results of \cite[Theorems B and C]{woods-prime-quotients}, which allow us to understand the prime ideals of $kH$, along with the conjugation action of $G$, in terms of the corresponding information for $k'[[H/\Delta^+]]$ (for various finite field extensions $k'/k$). Now, as $\Delta^+(G/\Delta^+) = 1$ and $H/\Delta^+ = \fn(G/\Delta^+)$, we are back in the previous case. See Proposition \ref{propn: main prime extension theorem, Delta+ not equal to 1} for details.

\newpage

\section{$p$-valuations and crossed products}

\subsection{Preliminaries on $p$-valuations}

\begin{defn}Recall from \cite[III, 2.1.2]{lazard} that a \emph{$p$-valuation} on a group $G$ is a function $\omega: G\to \mathbb{R}\cup \{\infty\}$ satisfying:
\begin{itemize}
\item $\omega(xy^{-1}) \geq \min\{\omega(x), \omega(y)\}$ for all $x,y\in G$
\item $\omega([x,y]) \geq \omega(x) + \omega(y)$ for all $x,y\in G$
\item $\omega(x) = \infty$ if and only if $x = 1$
\item $\omega(x) > \frac{1}{p-1}$ for all $x\in G$
\item $\omega(x^p) = \omega(x) + 1$ for all $x\in G$.
\end{itemize}
Throughout this paper we will often be considering several $p$-valuations admitted by a group $G$, so to clarify we may refer to $G$ together with a $p$-valuation $\omega$ as the $p$-valued group $(G,\omega)$ (though when the $p$-valuation in question is clear from context, we will simply write $G$).

Given a $p$-valuation $\omega$ on a group $G$, we may write
\begin{align*}
G_{\omega, \lambda} := G_\lambda &:= \omega^{-1} ([\lambda, \infty]),\\
G_{\omega, \lambda^+} := G_{\lambda^+} &:= \omega^{-1} ((\lambda, \infty])
\end{align*}
and define the \emph{graded group}
$$\gr_\omega G := \bigoplus_{\lambda\in \mathbb{R}} G_\lambda / G_{\lambda^+}.$$
Then each element $1\neq x\in G$ has a \emph{principal symbol}
$$\gr_\omega (x) := x G_{\mu^+} \in G_\mu/G_{\mu^+}\leq \gr_\omega G,$$
where $\mu$ is defined such that $\mu = \omega(x)$.
\end{defn}

\begin{rk}
Let $(G, \omega)$ be a $p$-valued group, and $N$ an arbitrary subgroup of $G$. Then $(N, \omega|_N)$ is $p$-valued. Moreover, if $G$ has finite rank \cite{lazard}, then so does $N$; and if $G$ is complete with respect to $\omega$ and $N$ is a closed subgroup of $G$, then $N$ is complete with respect to $\omega|_N$.
\end{rk}

\begin{defn}\label{defn: quotient p-val}
Given an arbitrary complete $p$-valued group $(G,\omega)$ of finite rank, and a closed isolated normal subgroup $K$ (i.e. a closed normal subgroup $K$ such that $G/K$ is torsion-free), we may define the \emph{quotient $p$-valuation} $\Omega$ induced by $\omega$ on $G/K$ as follows:
$$\Omega(gK) = \sup_{k\in K}\{\omega(gk)\}.$$
This is defined by Lazard, but the definition is spread across several results, so we collate them here for convenience. The definition in the case of filtered modules is \cite[I, 2.1.7]{lazard}, and is modified to the case of filtered groups in \cite[the remark after II, 1.1.4.1]{lazard}. The specialisation from filtered groups to $p$-saturable groups is done in \cite[III, 3.3.2.4]{lazard}, where it is proved that $\Omega$ is indeed still a $p$-valuation on $G/K$; and the general case is stated in \cite[III, 3.1.7.6]{lazard}, and eventually proved in \cite[IV, 3.4.2]{lazard}.
\end{defn}

As a partial inverse to the above process of passing to a quotient $p$-valuation, we prove the following general result about ``lifting" $p$-valuations from torsion-free quotients.

\begin{thm}\label{thm: lifting p-valuations}
Let $G$ be a complete $p$-valued group of finite rank, and $N$ a closed isolated orbital (hence normal) subgroup of $G$. Suppose we are given two functions
$$\alpha, \beta: G\to \Rinfty,$$
such that $\alpha$ is a $p$-valuation on $G$, and $\beta$ factors through a $p$-valuation on $G/N$, i.e.
$$\overline{\beta}: G/N \to \Rinfty.$$
Then $\omega = \inf\{\alpha, \beta\}$ is a $p$-valuation on $G$.
\end{thm}

\begin{proof}
$\alpha$ and $\beta$ are both filtrations on $G$ (in the sense of \cite[II, 1.1.1]{lazard}), and so by \cite[II, 1.2.10]{lazard}, $\omega$ is also a filtration. Following \cite[III, 2.1.2]{lazard}, for $\omega$ to be a $p$-valuation, we need to check the following three conditions:
\begin{itemize}
\item[(i)] \emph{$\omega(x) < \infty$ for all $x\in G$, $x\neq 1$.}

This follows from the fact that $\alpha$ is a $p$-valuation, and hence $\alpha(x) < \infty$ for all $x\in G$, $x\neq 1$.
\item[(ii)] \emph{$\omega(x) > (p-1)^{-1}$ for all $x\in G$.}

This follows from the fact that $\alpha(x) > (p-1)^{-1}$ and $\beta(x) > (p-1)^{-1}$ for all $x\in G$ by definition.
\item[(iii)] \emph{$\omega(x^p) = \omega(x) + 1$ for all $x\in G$.}

Take any $x\in G$. As $\alpha$ is a $p$-valuation, we have by definition that $\alpha(x^p) = \alpha(x) + 1$.

If $x\in N$, this alone is enough to establish the condition, as $\omega|_N = \alpha|_N$ (since $\beta(x) = \infty$).

Suppose instead that $x\in G\setminus N$. Then, as $N$ is assumed \emph{isolated} orbital in $G$, we also have $x^p\in G\setminus N$, so by definition of $\beta$ we have
$$\beta(x^p) = \overline{\beta}((xN)^p) = \overline{\beta}(xN) + 1 = \beta(x) + 1,$$
with the middle equality coming from the fact that $\overline{\beta}$ is a $p$-valuation. Now it is clear that $\omega(x^p) = \omega(x) + 1$ by definition of $\omega$.\qedhere
\end{itemize}
\end{proof}

Finally, the following function will be crucial. (It is in fact a $p$-valuation, but we delay the proof of this fact until Lemma \ref{lem: properties of (t,p)-filtration}.)

\begin{defn}\label{defn: (t,p)-filtration}
Let $A$ be a free abelian pro-$p$ group of rank $d > 0$ (here written multiplicatively). Choose a real number $t > (p-1)^{-1}$. Then the \emph{$(t,p)$-filtration} on $A$ is the function $\omega: A\to \Rinfty$ defined by
$$\omega(x) = t+n,$$
where $n$ is the non-negative integer such that $x\in A^{p^n} \setminus A^{p^{n-1}}$. (By convention, $\omega(1) = \infty.$)
\end{defn}

\subsection{Ordered bases}

\begin{defn}
Recall from \cite[4.2]{ardakovInv} that an \emph{ordered basis} for a $p$-valued group $(G,\omega)$ is a set $\{g_1, \dots, g_e\}$ of elements of $G$ such that every element $x\in G$ can be uniquely written as the (ordered) product $$x = \prod_{1\leq i\leq e} g_i^{\lambda_i}$$ for some $\lambda_i \in \mathbb{Z}_p$, and $$\omega(x) = \inf_{1\leq i\leq e} \{\omega(g_i) + v_p(\lambda_i)\},$$ where $v_p$ is the usual $p$-adic valuation on $\mathbb{Z}_p$. (Note that an ordered basis for $(G,\omega)$ need not be an ordered basis for $(G,\omega')$ for another $p$-valuation $\omega'$.)

As in \cite{ardakovInv}, we will often write $$\mathbf{g}^\lambda := \prod_{1\leq i\leq e} g_i^{\lambda_i}$$ as shorthand, where $\lambda\in \mathbb{Z}_p^e$.
\end{defn}

We now show that the function given in Definition \ref{defn: (t,p)-filtration} is indeed a $p$-valuation, and demonstrate some of its properties.

\begin{lem}\label{lem: properties of (t,p)-filtration}
Let $A$ and $t$ be as in Definition \ref{defn: (t,p)-filtration}.
\begin{itemize}
\item[(i)] The $(t,p)$-filtration $\omega$ is a $p$-valuation on $A$.
\item[(ii)] Suppose we are given a $\mathbb{Z}_p$-module basis $B = \{a_1, \dots, a_d\}$ for $A$, and a $p$-valuation $\alpha$ on $A$ satisfying $\alpha(a_1) = \dots = \alpha(a_d) = t$. Then $\alpha$ is the $(t,p)$-filtration on $A$, and $B$ is an ordered basis for $(A,\alpha)$.
\item[(iii)] The $(t,p)$-filtration $\omega$ is completely invariant under automorphisms of $A$, i.e. the subgroups $A_{\omega, \lambda}$ and $A_{\omega, \lambda^+}$ are characteristic in $A$.
\end{itemize}
\end{lem}

\begin{proof}
$ $

\begin{itemize}
\item[(i)] This is a trivial check from the definition \cite[III, 2.1.2]{lazard}.
\item[(ii)] By \cite[III, 2.2.4]{lazard}, we see that
$$\alpha(a_1^{\lambda_1}\dots a_d^{\lambda_d}) = t + \inf_{1\leq i\leq d} \{ v_p(\lambda_i) \},$$
which is precisely the $(t,p)$-filtration.
\item[(iii)] The subgroups $A^{p^n}$ are clearly characteristic in $A$.\qedhere
\end{itemize}
\end{proof}

\begin{rk}
The $(t,p)$-filtration as defined here is equivalent to the definition given in \cite[II, 3.2.1]{lazard} for free abelian pro-$p$ groups of finite rank.
\end{rk}

Recall from \cite[Definitions 1.1 and 1.4]{woods-struct-of-G} that a closed subgroup $H$ of a profinite group $G$ is \emph{($G$-)orbital} if it has finitely many $G$-conjugates, and \emph{isolated orbital} if any $G$-orbital $H' \gneq H$ satisfies $[H':H] = \infty$.

The following is a general property of ordered bases.

\begin{lem}\label{lem: extending ordered bases}
Let $(G,\omega)$ be a complete $p$-valued group of finite rank, and $N$ a closed isolated normal subgroup of $G$. Then there exist sets $B_N \subseteq B_G$ such that $B_N$ is an ordered basis for $(N,\omega|_N)$ and $B_G$ is an ordered basis for $(G,\omega)$.
\end{lem}

\begin{proof}
This was established in \cite[proof of Lemma 8.5(a)]{ardakovInv}.
\end{proof}

\begin{rk}
It may be helpful to think of this as follows:
\begin{align*}
B_G = \big\{ \underbrace{x_1, \dots, x_r}_{B_{G/N}}, \;\; \underbrace{x_{r+1}, \dots, x_s}_{B_{N}} \big\},
\end{align*}
where $B_{G/N} = B_G \setminus B_N$ is in fact some appropriate preimage in $G$ of any ordered basis for $(G/N,\Omega)$, where $\Omega$ is the quotient $p$-valuation.
\end{rk}

\begin{lem}\label{lem: o.b. doesn't change}
Let $(G,\alpha)$ be a complete $p$-valued group of finite rank, and $N$ a closed isolated orbital (hence normal) subgroup of $G$. Take also a $p$-valuation $\overline{\beta}$ on $G/N$. Suppose we are given sets
\begin{align*}
B_G = \big\{ \underbrace{x_1, \dots, x_r}_{B_{G/N}}, \;\; \underbrace{x_{r+1}, \dots, x_s}_{B_{N}} \big\},
\end{align*}
such that
\begin{itemize}
\item $B_N$ is an ordered basis for $(N, \alpha|_N)$,
\item $B_G$ is an ordered basis for $(G, \alpha)$, and
\item the image in $G/N$ of $B_{G/N}$ is an ordered basis for $(G/N, \overline{\beta})$.
\end{itemize}
In the notation of Theorem \ref{thm: lifting p-valuations}, write $\beta$ for the composite of $G\to G/N$ with $\overline{\beta}$, and form the $p$-valuation $\omega = \inf\{\alpha, \beta\}$ for $G$.

Then $B_G$ is an ordered basis for $(G, \omega)$.
\end{lem}

\begin{proof}
We need only check that
$$\omega(\mathbf{x}^\lambda) = \inf_{1\leq i\leq s} \{\omega(x_i) + v_p(\lambda_i)\}$$
for any $\lambda\in\mathbb{Z}_p^s$. But we have by definition that
$$\alpha(\mathbf{x}^\lambda) = \inf_{1\leq i\leq s} \{\alpha(x_i) + v_p(\lambda_i)\},$$
$$\beta(\mathbf{x}^\lambda) = \inf_{1\leq i\leq r} \{\beta(x_i) + v_p(\lambda_i)\},$$
and the result follows trivially.
\end{proof}

\subsection{Separating a free abelian quotient}

Results in later sections will require the existence of a $p$-valuation on an appropriate group $G$ satisfying a certain technical property, which we can now finally state:

\begin{defn}\label{defn: omegaprop}
Let $(G,\omega)$ be a complete $p$-valued group of finite rank, and $L$ a closed isolated normal subgroup containing $[G,G]$ (and hence containing the isolated derived subgroup $G'$, which was defined and written as $G^{(1)}$ in \cite[Theorem B]{woods-struct-of-G}). We will say that $\omega$ satisfies property (\omegaprop{L}) if there is an ordered basis $\{g_{d+1}, \dots, g_e\}$ for $(L,\omega|_L)$, contained in an ordered basis $\{g_1, \dots, g_e\}$ for $(G,\omega)$ (e.g. constructed by Lemma \ref{lem: extending ordered bases}), such that the following hold:
\begin{flalign}
\begin{aligned}
\\
\txt{and, for all $\ell\in L$},
\end{aligned}
&&
\begin{rcases}
\omega(g_1) = \dots = \omega(g_d),\quad\\
\omega(g_1) < \omega(\ell).\quad
\end{rcases}&&
\tag{\omegaprop{L}}
\end{flalign}
\end{defn}

\begin{rk}
In the notation of the above definition, suppose that $\omega$ satisfies $\omega(g_1) < \omega(\ell)$ for all $\ell \in L$. Then, by our earlier remarks, we note that the condition (\omegaprop{L}) is equivalent to the statement that the quotient $p$-valuation induced by $\omega$ on $G/L$ is the $(t,p)$-filtration for $t := \omega(g_1)$.
\end{rk}

\begin{defn}
Following \cite[III, 2.1.2]{lazard}, we will say that a group $G$ is \emph{$p$-valuable} if there exists a $p$-valuation $\omega$ for $G$, and $G$ is complete with respect to $\omega$ and has finite rank.
\end{defn}

\begin{lem}\label{cor: there exists a nice omega}
Let $G$ be a nilpotent $p$-valuable group, and $L$ a closed isolated normal subgroup containing $G'$. Then there exists some $p$-valuation $\omega$ for $G$ satisfying (\omegaprop{L}).
\end{lem}

\begin{proof}
Let $\alpha$ be a $p$-valuation on $G$. Take an ordered basis $\{g_{d+1}, \dots, g_e\}$ for $(L, \alpha|_L)$ and extend it to an ordered basis $\{g_1, \dots, g_e\}$ for $(G,\alpha)$ by Lemma \ref{lem: extending ordered bases}. Fix a number $t$ satisfying
\begin{align*}
(p-1)^{-1} < t \leq \inf_{1\leq i\leq e} \alpha(g_i).
\end{align*}
Applying Theorem \ref{thm: lifting p-valuations} with $N = L$ and $\overline{\beta}$ the $(t, p)$-filtration on $G/L$, we see that $\omega = \inf\{\alpha, \beta\}$ is a $p$-valuation for $G$; and by Lemma \ref{lem: o.b. doesn't change}, $\{g_1, \dots, g_e\}$ is still an ordered basis for $(G,\omega)$, so we can check easily that $\omega$ satisfies (\omegaprop{L}) by construction.
\end{proof}


\begin{rk}
In fact, by analysing the construction in Theorem \ref{thm: lifting p-valuations}, we can see that we have shown something stronger: that any $p$-valuation $\alpha$ may be refined to such an $\omega$ satisfying $\omega|_L = \alpha|_L$. But we will not use this fact in this paper.
\end{rk}

\begin{rk}
If $\omega$ satisfies (\omegaprop{L}) as above, write $t := \omega(g_1)$. Then, for any automorphism $\sigma$ of $G$ and any $1\leq i\leq d$, we have $$\omega(\sigma(g_i)) = t.$$ This follows from Lemma \ref{lem: properties of (t,p)-filtration}(iii). Indeed, by construction, we have $G_t = G$, and $G_{t^+} = G^p\cdot L$, an open normal subgroup; and when $L$ is characteristic, $G_{t^+}$ is characteristic.
\end{rk}

Now let $G$ be a $p$-valuable group with fixed $p$-valuation $\omega$, and let $\sigma\in \Aut(G)$. In this subsection and the next, we seek to establish conditions under which a given automorphism $\sigma$ of $G$ will preserve the ``dominant" part of certain elements $x\in G$ (with respect to $\omega$). That is, we are looking for a condition under which
$$\gr_\omega (\sigma(x)) = \gr_\omega(x).$$
Clearly it is necessary and sufficient that the following holds:
\begin{align}
\omega(\sigma(x)x^{-1}) > \omega(x).\label{eqn: sigma-1 increases omega}
\end{align}

The results of this paper rely on our ability to invoke the following technical result.

\begin{thm}\label{thm: M is in Gamma_1}
Let $G$ be a $p$-valuable group, and let $L$ be a proper closed isolated orbital (hence normal) subgroup containing $[G,G]$, so that we have an isomorphism $\varphi: G/L \to \mathbb{Z}_p^d$ for some $d\geq 1$. Write $q: G\to G/L$ for the natural quotient map.

Choose a $\mathbb{Z}_p$-basis $\{e_1, \dots, e_d\}$ for $\mathbb{Z}_p^d$. For each $1\leq i\leq d$, fix an element $g_i\in G$ with $\varphi\circ q(g_i) = e_i$. Fix an automorphism $\sigma$ of $G$ preserving $L$, so that $\sigma$ induces an automorphism $\overline{\sigma}$ of $G/L$, and hence an automorphism $\hat{\sigma} = \varphi\circ\overline{\sigma}\circ\varphi^{-1}$ of $\mathbb{Z}_p^d$. Let $M_\sigma$ be the matrix of $\hat{\sigma}$ with respect to the basis $\{e_1, \dots, e_d\}$.

Suppose there exists some $p$-valuation $\omega$ on $G$ with the following properties:
\begin{itemize}
\item[(i)] (\ref{eqn: sigma-1 increases omega}) holds for all $x\in \{g_1, \dots, g_d\}$,
\item[(ii)] $\omega(g_1) = \dots = \omega(g_d) (= t$, say$)$,
\item[(iii)] $\omega(\ell) > t$ for all $\ell\in L$.
\end{itemize}
Then $M_\sigma - 1 \in pM_d(\mathbb{Z}_p)$.
\end{thm}

\begin{rk}
Conditions (ii) and (iii) are just the statement that $\omega$ satisfies (\omegaprop{L}).
\end{rk}

\begin{proof}
Define the function $\Omega: \mathbb{Z}_p^d\to \Rinfty$ by
$$\Omega\circ\varphi(gL) = \sup_{\ell\in L}\{\omega(g\ell)\}.$$
By the remarks made in Definition \ref{defn: quotient p-val}, $\Omega$ is in fact a $p$-valuation.

By assumption (iii), we see that, for each $1\leq i\leq d$ and any $\ell\in L$, we have $\omega(g_i) = \omega(g_i\ell)$, so that
$$\Omega(e_i) = \Omega\circ\varphi(g_i L) = \sup_{\ell\in L}\{\omega(g_i\ell)\} = \omega(g_i),$$
so by assumption (ii), $\Omega(e_i) = t$. Hence, by Lemma \ref{lem: properties of (t,p)-filtration}(ii), $\Omega$ must be the $(t,p)$-filtration on $\mathbb{Z}_p^d$. Now, by assumption (i), we have
\begin{align*}
\Omega(\hat{\sigma}(x) - x) > t
\end{align*}
for all $x\in\{e_1, \dots, e_d\}$, and hence, as $\Omega-t$ takes integer values (by Definition \ref{defn: (t,p)-filtration}),
\begin{align*}
\Omega(\hat{\sigma}(x) - x) \geq t+1,
\end{align*}
and so $\hat{\sigma}(x) - x\in p\mathbb{Z}_p^d$ for each $x\in\{e_1, \dots, e_d\}$, which is what we wanted to prove.
\end{proof}

\begin{rk}
With Lemma \ref{lem: properties of (t,p)-filtration} in mind, we note the following: suppose $\omega$ satisfies hypothesis (iii) of Theorem \ref{thm: M is in Gamma_1}. Then hypothesis (ii) is equivalent to the statement that the quotient filtration induced by $\omega$ on $G/L$ is actually the $(t,p)$-filtration on $G/L$.
\end{rk}


\subsection{Invariance under the action of a crossed product}

\begin{defn}\label{defn: cpnormal}
Let $R$ be a ring, and fix a subgroup $G\leq R^\times$; let $F$ be a group. Fix a crossed product $$S = R\cp{\sigma}{\tau} F.$$ Consider the following properties that this crossed product may satisfy:
\begin{align}
&\txt{The image $\sigma(F)$ normalises $G$, i.e. $x^{\sigma(f)} \in G$ for all $x\in G, f\in F$.}\tag{\cpnormal{G}}\\
&\txt{The image $\tau(F,F)$ normalises $G$.}\tag{\taunormal{G}}\\
&\txt{The image $\tau(F,F)$ is a subset of $G$.}\tag{\tausub{G}}
\end{align}
In the case when $G$ is $p$-valuable, consider the set of $p$-valuations of $G$. Then $\Aut(G)$ acts on this set as follows:
$$(\varphi\cdot\omega)(x) = \omega(x^\varphi).$$
When $S$ satisfies (\taunormal{G}), $\tau(F,F)\subseteq G$, so we get a map $\rho: \tau(F,F)\to \mathrm{Inn}(G)$ (with elements of $G$ acting by conjugation), so we will also consider the following property:
\begin{align}
&\txt{Every $p$-valuation $\omega$ of $G$ is invariant under elements of $\tau(F,F)$.}\tag{\tauinv{G}}
\end{align}
\end{defn}

\begin{lem}\label{lem: cp properties}
In the notation above:
\begin{itemize}
\item[(i)] If $S$ satisfies (\cpnormal{G}), then $S$ satisfies (\taunormal{G}).
\item[(ii)] If $S$ satisfies (\tausub{G}), then $S$ satisfies (\taunormal{G}).
\item[(iii)] If $S$ satisfies (\tausub{G}), then $S$ satisfies (\tauinv{G}).
\end{itemize}
\end{lem}

\begin{proof}
$ $

\begin{itemize}
\item[(i)] Note that $\rho\circ \tau(x,y) = \sigma(xy)^{-1}\sigma(x)\sigma(y)$.
\item[(ii)] Obvious.
\item[(iii)] By (ii), we see that $S$ satisfies (\taunormal{G}), so it makes sense to consider (\tauinv{G}).

Let $\omega$ be a $p$-valuation of $G$, and take $t\in \tau(F,F)$. As $S$ satisfies (\tausub{G}), we actually have $t\in G$. Then, for any $x\in G$, we have
\begin{align*}
(t\cdot\omega)(x) &= \omega(x^t)\\
&= \omega(t^{-1} x t)\\
&= \omega(x \cdot [x,t]) \\
&\geq \min\{ \omega(x), \omega([x,t]) \} = \omega(x),
\end{align*}
and so (by symmetry) $\omega(t^{-1}xt) = \omega'(x)$.
\end{itemize}
\end{proof}

\begin{defn}\label{defn: recall 2-cocycle twist}
Recall, from \cite[Definition 5.4]{woods-prime-quotients}, that if we have a fixed crossed product
\begin{align}
S = R\cp{\sigma}{\tau} F\label{eqn: a cp}
\end{align}
and a 2-cocycle $$\alpha\in Z^2_\sigma(F, Z(R^\times)),$$ then we may define the ring $$S_\alpha = R\cp{\sigma}{\tau\alpha} F,$$ the \emph{2-cocycle twist (of $R$, by $\alpha$, with respect to the decomposition (\ref{eqn: a cp}))}.
\end{defn}

\begin{lem}\label{lem: cpnormality preserved under twisting}
Continuing with the notation above,
\begin{itemize}
\item[(i)] $S$ satisfies (\cpnormal{G}) if and only if $S_\alpha$ satisfies (\cpnormal{G}).
\item[(ii)] $S$ satisfies (\tauinv{G}) if and only if $S_\alpha$ satisfies (\tauinv{G}).
\end{itemize}
\end{lem}

\begin{proof}
$ $

\begin{itemize}
\item[(i)] Trivial from Definitions \ref{defn: cpnormal} and \ref{defn: recall 2-cocycle twist}.
\item[(ii)] As $\alpha(F,F) \subseteq Z(R)^\times$, conjugation by $\alpha$ is the identity automorphism on $G$.\qedhere
\end{itemize}
\end{proof}

These properties will be interesting to us later as they will allow us to invoke the following lemma:

\begin{lem}\label{lem: group action on p-valuations}
If $S$ satisfies (\cpnormal{G}), then, given any $g\in F$ and $p$-valuation $\omega$ on $G$, the function $g\cdot\omega$ given by $$(g\cdot\omega)(x) = \omega(x^{\sigma(g)})$$ is again a $p$-valuation on $G$. If, further, $S$ satisfies (\tauinv{G}), then this is a \emph{group} action of $F$ on the set of $p$-valuations of $G$.
\end{lem}

\begin{proof}
If $x\in G$, then $x^{\sigma(g)}\in G$ because $S$ satisfies (\cpnormal{G}), so it makes sense to consider $\omega(x^{\sigma(g)})$. The definition above does indeed give a group action when $S$ satisfies (\tauinv{G}), as, for all $g,h\in F$,
\begin{align*}
(g\cdot(h\cdot \omega))(x) & = h\cdot \omega(x^{\sigma(g)})\\
&= \omega(x^{\sigma(g)\sigma(h)})\\
&= \omega(x^{\sigma(gh)\tau(g,h)})\\
&= \omega(x^{\sigma(gh)})&\txt{by (\tauinv{G})}\\
&= (gh\cdot\omega)(x).&&\qedhere
\end{align*}
\end{proof}

The following lemma will allow us to prove the existence of a sufficiently ``nice" $p$-valuation.

\begin{lem}\label{lem: there exists an F-stable omega satisfying omegaprop}
Suppose $S$ satisfies (\cpnormal{G}) and (\tauinv{G}), so that $\sigma$ induces an action of $F$ on the set of $p$-valuations on $G$ as in the above lemma. Let $\omega$ be a $p$-valuation on $G$. If the $F$-orbit of $\omega$ is finite, then $\omega'(x) = \inf_{g\in F}(g\cdot \omega)(x)$ defines an $F$-invariant $p$-valuation on $G$.

Furthermore, if $L$ is a closed isolated characteristic subgroup of $G$ containing $G'$, and $\omega$ satisfies (\omegaprop{L}) (as in Definition \ref{defn: omegaprop}), then $\omega'$ satisfies (\omegaprop{L}).
\end{lem}

\begin{proof}
The function $\omega'$ satisfies condition \cite[III, 2.1.2.2]{lazard}, since the $F$-orbit of $\omega$ is finite, and is hence a $p$-valuation that is $F$-stable by the remark in \cite[III, 2.1.2]{lazard}.

Suppose $\omega$ satisfies (\omegaprop{L}). That is, for some $t > (p-1)^{-1}$, $\omega$ induces the $(t,p)$-filtration on $G/L$, and $\omega(\ell) > t$ for all $\ell\in L$. But, given any $g\in F$, clearly $g\cdot\omega$ still induces the $(t,p)$-filtration on $G/L$ by Lemma \ref{lem: properties of (t,p)-filtration}(iii), and $(g\cdot\omega)(\ell) = \omega(\ell^{\sigma(g)}) > t$, since $\ell^{\sigma(g)}\in L$ as $L$ is characteristic. Taking the infimum over the finitely many distinct $g\cdot\omega$, $g\in F$, shows that $\omega'$ also satisfies (\omegaprop{L}).
\end{proof}

Recall the \emph{finite radical} $\Delta^+ = \Delta^+(G)$ from \cite[Definition 1.2]{woods-struct-of-G}.

\begin{defn}\label{defn: recall standard cp}
Let $G$ be an arbitrary compact $p$-adic analytic group with $\Delta^+ = 1$, $H$ an open normal subgroup of $G$, $F = G/H$, and $P$ a faithful $G$-stable ideal of $kH$. Recall from \cite[Definition 5.11]{woods-prime-quotients} that the crossed product decomposition $$kG/PkG = kH/P \cp{\sigma}{\tau} F$$ is \emph{standard} if the basis $\overline{F}$ is a subset of the image of the map $G\hookrightarrow (kG/PkG)^\times$.
\end{defn}

\begin{lem}\label{lem: omega in standard cps}
Suppose that $kG/PkG = kH/P \cp{\sigma}{\tau} F$ is a standard crossed product decomposition. Take any $\alpha\in Z^2_\sigma(F, Z((kH/P)^\times))$, and form the central 2-cocycle twist $$(kG/PkG)_\alpha := kH/P \cp{\sigma}{\tau\alpha} F$$ with respect to this decomposition \cite[Definition 5.4]{woods-prime-quotients}.

Consider $H$ as a subgroup of $(kH/P)^\times$: then conjugation by elements of $\overline{G}$ inside $((kG/PkG)_\alpha)^\times$ induces a group action of $F$ on the set of $p$-valuations of $H$, as in Lemma \ref{lem: group action on p-valuations}.
\end{lem}

\begin{rk}
As the crossed product notation suggests, this lemma simply says that the action of $F$ on $H$, via $\sigma$, is unchanged after applying $(-)_\alpha$.
\end{rk}

\begin{proof}
As the decomposition is standard, $kG/PkG$ trivially satisfies both (\cpnormal{H}) (as $H$ is normal in $G$) and (\tausub{H}). By Lemma \ref{lem: cp properties}(iii), $kG/PkG$ also satisfies (\tauinv{H}). Now Lemma \ref{lem: cpnormality preserved under twisting} shows that $(kG/PkG)_\alpha$ also satisfies (\cpnormal{H}) and (\tauinv{H}), so that $\sigma$ induces a group action of $F$ on the $p$-valuations of $H$ inside $(kG/PkG)_\alpha$ by Lemma \ref{lem: group action on p-valuations}.
\end{proof}

Let $L$ be a closed isolated characteristic subgroup of $H$ containing $[H,H]$.

\begin{cor}\label{cor: there exists an F-stable omega satisfying omegaprop}
With notation as above, we can find an $F$-stable $p$-valuation $\omega$ on $H$ satisfying (\omegaprop{L}).
\end{cor}

\begin{proof}
This now follows immediately from Lemmas \ref{cor: there exists a nice omega} and \ref{lem: there exists an F-stable omega satisfying omegaprop}.
\end{proof}

\textit{Proof of Theorem B.} This follows from Corollary \ref{cor: there exists an F-stable omega satisfying omegaprop}. \qed

\section{A graded ring}\label{section: graded ring}

\subsection{Generalities on ring filtrations}\label{subsection: ring filtrations, w, f}

\begin{defn}
Recall that a \emph{filtration} $v$ on the ring $R$ is a function $v:R\to\Rinfty$ satisfying, for all $x, y\in R$,
\begin{itemize}
\item $v(x+y) \geq \min\{v(x),v(y)\},$
\item $v(xy) \geq v(x)+v(y),$
\item $v(0) = \infty, v(1) = 0.$
\end{itemize}
If in addition we have $v(xy) = v(x)+v(y)$ for all $x,y\in R$, then $v$ is a \emph{valuation} on $R$.
\end{defn}

First, a basic property of ring filtrations.

\begin{lem}\label{lem: basic filtration property}
Suppose $v$ is a filtration on $R$ which takes non-negative values, i.e. $v(R)\subseteq [0,\infty]$, and let $u\in R^\times$. Then $v(ux) = v(xu) = v(x)$ for all $x\in R$.
\end{lem}

\begin{proof}
By the definition of $v$, we have $0 = v(1) = v(uu^{-1}) \geq v(u) + v(u^{-1})$. As $v(u) \geq 0$ and $v(u^{-1}) \geq 0$, we must have $v(u) = 0 = v(u^{-1})$. Then
$$v(x) = v(u^{-1}ux) \geq v(u^{-1}) + v(ux) = v(ux) \geq v(u) + v(x) = v(x),$$
from which we see that $v(x) = v(ux)$; and by a symmetric argument, we also have $v(xu) = v(x)$.
\end{proof}

\textbf{We will fix the following notation for this subsection.}

\begin{notn}
Let $G$ be a $p$-valuable group equipped with the fixed $p$-valuation $\omega$, and $k$ a field of characteristic $p$. Take an ordered basis $\{g_1, \dots, g_d\}$ for $G$, and write $b_i = g_1 - 1 \in kG$ for all $1\leq i\leq d$. As in \cite{ardakovInv}, we make the following definitions:
\begin{itemize}
\item for each $\alpha\in \mathbb{N}^d$, $\mathbf{b}^\alpha$ means the (ordered) product $b_1^{\alpha_1} \dots b_d^{\alpha_d}\in kG$,
\item for each $\alpha\in \mathbb{Z}_p^d$, $\mathbf{g}^\alpha$ means the (ordered) product $g_1^{\alpha_1} \dots g_d^{\alpha_d}\in G$,
\item for each $\alpha\in \mathbb{N}^d$, $\langle \alpha, \omega(\mathbf{g}) \rangle$ means $\displaystyle \sum_{i=1}^d \alpha_i \omega(g_i)$,
\item the canonical ring homomorphism $\mathbb{Z}_p\to k$ will sometimes be left implicit, but will be denoted by $\iota$ when necessary for clarity.
\end{itemize}
\end{notn}

\begin{defn}\label{defn: valuation w}
With notation as above, let $w$ be the valuation on $kG$ defined in \cite[6.2]{ardakovInv}, given by
$$\sum_{\alpha\in \mathbb{N}^d} \lambda_\alpha \mathbf{b}^\alpha \mapsto \inf_{\alpha\in \mathbb{N}^d} \big\{\langle \alpha, \omega(\mathbf{g})\rangle\; \big| \; \lambda_\alpha \neq 0\big\}.$$
Note that, in light of this formula \cite[Corollary 6.2(b)]{ardakovInv}, and by the construction \cite[III, 2.3.3]{lazard} of $w$, it is clear that the value of $w$ is in fact independent of the ordered basis chosen. In particular, if $\varphi$ is an automorphism of $G$, then $\{g_1^\varphi, \dots, g_d^\varphi\}$ is another ordered basis of $G$; hence if $\omega$ is $\varphi$-stable (in the sense that $\omega(g^\varphi) = \omega(g)$ for all $g\in G$), then $w$ is $\varphi$-stable (in the sense that $w(x^{\hat\varphi}) = w(x)$ for all $x\in kG$, where $\hat\varphi$ here denotes the natural extension of $\varphi$ to $kG$, obtained by the universal property \cite[Lemma 2.2]{woods-prime-quotients}).
\end{defn}

We will need the following result:

\begin{lem}\label{lem: alperin binomial}
Let
$$b = b_0 + b_1 p + b_2 p^2 + \dots \in \mathbb{Z}_p,$$
$$n = n_0 + n_1 p + n_2 p^2 + \dots + n_s p^s\in \mathbb{N},$$
where all $b_i, n_i \in \{0,1,\dots,p-1\}$. Then
$$\binom{b}{n} \equiv \prod_{i=0}^s \binom{b_i}{n_i} \mod p.$$
\end{lem}

\begin{proof}
See e.g. \cite[Theorem]{alperin}.
\end{proof}

\begin{cor}\label{cor: binomial stuff}
Let $b\in \mathbb{Z}_p$, $n\in \mathbb{N}$. If
\begin{equation}\label{eqn: binomial}
\displaystyle v_p\left(\binom{b}{n}\right) = 0,
\end{equation}
then $v_p(b) \leq v_p(n)$. Further, for fixed $b\in \mathbb{Z}_p$,
$$\inf\left\{ n\in \mathbb{N} \;\middle|\; v_p\left(\binom{b}{n}\right) = 0\right\} = p^{v_p(b)}.$$
\end{cor}

\begin{proof}
From Lemma \ref{lem: alperin binomial} above, we can see that $$\binom{b}{n} \equiv 0 \mod p$$ if and only if, for some $0\leq i\leq s$, $$\binom{b_i}{n_i} = 0,$$ which happens if and only if one of the pairs $(b_i, n_i)$ for $0\leq i\leq s$ has $b_i = 0 \neq n_i$. Hence, to ensure that this does not happen, we must have $v_p(b) \leq v_p(n)$. It is clear from Lemma \ref{lem: alperin binomial} that $\displaystyle n = p^{v_p(b)}$ satisfies (\ref{eqn: binomial}), and is the least $n\in\mathbb{N}$ with $v_p(b) \leq v_p(n)$.
\end{proof}

\begin{thm}\label{thm: w(x-1) = omega(x) attempt 2}
Take any $x \in G$, and $t = \inf \omega(G)$. Then $w(x - 1) > t$ implies $\omega(x) > t$.
\end{thm}

\begin{proof}
Write $x = \mathbf{g}^\alpha$. In order to show that $\omega(\mathbf{g}^\alpha) > t$, it suffices to show that $\omega(g_j) + v_p(\alpha_j) > t$ for each $j$ (as there are only finitely many), and hence that $v_p(\alpha_j) \geq 1$ for all $j$ such that $\omega(g_j) = t$. This is equivalent to the claim that $p^{v_p(\alpha_j)} > 1$, which we will write as $p^{v_p(\alpha_j)} \omega(g_j) > t$ for all $j$ with $\omega(g_j) = t$.

Let $\beta^{(j)}$ be the $d$-tuple with $i$th entry $\delta_{ij} p^{v_p(\alpha_j)}$. Then, of course, $$\langle \beta^{(j)}, \omega(\mathbf{g})\rangle = p^{v_p(\alpha_j)} \omega(g_j),$$ and by Corollary \ref{cor: binomial stuff}, we have
$$\binom{\alpha}{\beta^{(j)}} \not\equiv 0 \mod p.$$

Now suppose that $w(\mathbf{g}^\alpha - 1) > t$. We perform binomial expansion in $kG$ to see that
\begin{align*}
\mathbf{g}^\alpha - 1 &= \prod_{1\leq i\leq d} (1+b_i)^{\alpha_i} - 1&\text{(ordered product)}\\
&= \sum_{\beta\in \mathbb{N}^d} \iota\binom{\alpha}{\beta} \mathbf{b}^\beta - 1\\
&= \sum_{\beta\neq 0} \iota\binom{\alpha}{\beta} \mathbf{b}^\beta,
\end{align*}
so that 
$$w(\mathbf{g}^\alpha - 1) = \inf\left\{ \langle \beta, \omega(\mathbf{g})\rangle \;\middle|\; \beta \neq 0, \binom{\alpha}{\beta} \not\equiv 0 \mod p \right\}.$$
So in particular, for all $j$ satisfying $\omega(g_j) = t$, we have
$$t < w(\mathbf{g}^\alpha - 1) \leq \langle \beta^{(j)}, \omega(\mathbf{g})\rangle = p^{v_p(\alpha_j)} \omega(g_j),$$
which is what we wanted to prove.
\end{proof}

%
%
%

\subsection{Constructing a suitable valuation}\label{subsection: Q', Q, filtrations}

Let $H$ be a nilpotent $p$-valuable group with centre $Z$. If $k$ is a field of characteristic $p$, and $\mathfrak{p}$ is a faithful prime ideal of $kZ$, then by \cite[Theorem 8.4]{ardakovInv}, the ideal $P := \mathfrak{p}kH$ is again a faithful prime ideal of $kH$.

\textbf{We will fix the following notation for this subsection.}

\begin{notn}
Let $G$ be a nilpotent-by-finite compact $p$-adic analytic group, with $\Delta^+ = 1$, and let $H = \fn(G)$ \cite[Definition 5.3]{woods-struct-of-G}, here a \emph{nilpotent} $p$-valuable radical, so that $\Delta = Z := Z(H)$ \cite[proof of Lemma 1.2.3(iii)]{woods-struct-of-G}. We will also write $F = G/H$.

Define $Q' = \mathbf{Q}(kZ/\mathfrak{p})$, the (classical) field of fractions of the (commutative) domain $kZ/\mathfrak{p}$, and $Q = Q' \tensor{kZ} kH$, a tensor product of $kZ$-algebras, which (as $P = \mathfrak{p}kH$) we may naturally identify with the (right) localisation of $kH/P$ with respect to $(kZ/\mathfrak{p}) \setminus \{0\}$ -- a subring of the Goldie ring of quotients $\mathbf{Q}(kH/P)$.

Suppose further that the prime ideal $\mathfrak{p}\lhd kZ$ is invariant under conjugation by elements of $G$.

Choose a crossed product decomposition
$$kG/PkG = kH/P\cp{\sigma}{\tau} F$$
which is \emph{standard} in the sense of the notation of Corollary \ref{cor: there exists an F-stable omega satisfying omegaprop}. Choose also any $\alpha\in Z^2_\sigma(F, Z((kH/P)^\times))$, and form as in \cite[Definition 5.4]{woods-prime-quotients} the central 2-cocycle twist
$$(kG/PkG)_\alpha = kH/P\cp{\sigma}{\tau\alpha} F.$$
Now the (right) divisor set $(kZ/\mathfrak{p}) \setminus \{0\}$ is $G$-stable by assumption, so by \cite[Lemma 37.7]{passmanICP}, we may define the partial quotient ring
\begin{align}\label{eqn: definition of Q*F}
R := Q\cp{\sigma}{\tau\alpha} F.
\end{align}
\end{notn}

Our aim in this subsection is to construct an appropriate filtration $f$ on the ring $R$. We will build this up in stages, following \cite{ardakovInv}. First, we define a finite set of valuations on $Q'$.

\begin{defn}\label{defn: filtration v_1}
In \cite[Theorem 7.3]{ardakovInv}, Ardakov defines a valuation on $\mathbf{Q}(kH/P)$; let $v_1$ be the restriction of this valuation to $Q'$, so that $v_1(x + \mathfrak{p}) \geq w(x)$ for all $x\in kZ$ (where $w$ is as in Definition \ref{defn: valuation w}).
\end{defn}

\begin{lem}
$\sigma$ induces a group action of $F$ on the set of valuations of $Q'$.
\end{lem}

\begin{proof}
Let $u$ be a valuation of $Q'$. $G$ acts on the set of valuations of $Q'$ as follows:
$$(g\cdot u)(x) = u(g^{-1}xg).$$
Clearly, if $g\in H$, then $g^{-1}xg = x$ (as $x\in \mathbf{Q}(kZ/\mathfrak{p})$ where $Z$ is the centre of $H$). Hence $H$ lies in the kernel of this action, and we get an action of $F$ on the set of valuations. By our choice of $\overline{F}$ as a subset of the image of $G$, this is the same as $\sigma$.
\end{proof}

Write $\{v_1, \dots, v_s\}$ for the $F$-orbit of $v_1$.

\begin{lem}\label{lem: v_i are indep}
The valuations $v_1, \dots, v_s$ are independent.
\end{lem}

\begin{proof}
The $v_i$ are all non-trivial valuations with value groups equal to subgroups of $\mathbb{R}$ by definition. Hence, by \cite[VI.4, Proposition 7]{bourbaki}, they have height 1.

They are also pairwise inequivalent: indeed, suppose $v_i$ is equivalent to $g\cdot v_i$ for some $g\in F$. Then by \cite[VI.3, Proposition 3]{bourbaki}, there exists a positive real number $\lambda$ with $v_i = \lambda (g\cdot v_i)$, and so $v_i = \lambda^n (g^n\cdot v_i)$ (as the actions of $\lambda$ and $g$ commute) for all $n$. But $F$ is a finite group: so, taking $n = o(g)$, we get $v_i = \lambda^n v_i$. As $v_i$ is non-trivial, we must have that $\lambda^n = 1$, and so $\lambda = 1$. So we may conclude, from \cite[VI.4, Proposition 6(c)]{bourbaki}, that the valuations $v_1, \dots, v_s$ are independent.
\end{proof}

\begin{defn}\label{defn: filtration v}
Let $v$ be the filtration of $Q'$ defined by $$\displaystyle v(x) = \inf_{1\leq i\leq s} v_i(x)$$ for each $x\in Q'$.
\end{defn}

\begin{lem}\label{lem: gr Q'}
$\displaystyle \gr_v Q' \cong \bigoplus_{i=1}^s \gr_{v_i} Q'$.
\end{lem}

\begin{proof}
The natural map
$$Q'_{v,\lambda} \to \bigoplus_{i=1}^s Q'_{v_i, \lambda} / Q'_{v_i, \lambda^+}$$
clearly has kernel $\displaystyle \bigcap_{i=1}^s Q'_{v_i, \lambda^+} = Q'_{v, \lambda^+}$, giving an injective map $\displaystyle \gr_v Q' \to \bigoplus_{i=1}^s \gr_{v_i} Q'$. The surjectivity of this map now follows from the Approximation Theorem \cite[VI.7.2, Th\'eor\`eme 1]{bourbaki}, as the $v_i$ are independent by Lemma \ref{lem: v_i are indep}.
\end{proof}

Next, we will extend the $v_i$ and $v$ from $Q'$ to $Q$, as in the proof of \cite[8.6]{ardakovInv}.

%
%
%

\begin{notn}
Continue with the notation above. Now, $H$ is $p$-valuable, and by Lemma \ref{lem: omega in standard cps}, $F$ acts on the set of $p$-valuations of $H$; hence, by Lemma \ref{lem: there exists an F-stable omega satisfying omegaprop} (or Corollary \ref{cor: there exists an F-stable omega satisfying omegaprop}), we may choose a $p$-valuation $\omega$ which is $F$-stable. Fix such an $\omega$, and construct the valuation $w$ on $kH$ from it as defined in Definition \ref{defn: valuation w}.

Let $\{y_{e+1}, \dots, y_d\}$ be an ordered basis for $Z$, and extend it to an ordered basis $\{y_1, \dots, y_d\}$ for $H$ as in Lemma \ref{lem: extending ordered bases} (noting that $Z$ is a closed isolated normal subgroup of $H$ by \cite[Lemma 8.4(a)]{ardakovInv}). For each $1\leq j\leq e$, set $c_j = y_j - 1$ inside the ring $kH/P$.

Recall from \cite[8.5]{ardakovInv} that elements of $Q$ may be written uniquely as $$\sum_{\gamma\in \mathbb{N}^e} r_\gamma \mathbf{c}^\gamma,$$ where $r_\gamma\in Q'$ and $\mathbf{c}^\gamma := c_1^{\gamma_1} \dots c_e^{\gamma_e}$, so that $Q \subseteq Q'[[c_1, \dots, c_e]]$ as a left $Q'$-module.
\end{notn}

\begin{defn}\label{defn: filtrations f_i}
For each $1\leq i\leq s$, as in \cite[proof of Theorem 8.6]{ardakovInv}, we will define the valuation $f_i: Q\to \Rinfty$ by $$f_i\left(\sum_{\gamma\in \mathbb{N}^e} r_\gamma \mathbf{c}^\gamma\right) = \inf_{\gamma\in \mathbb{N}^e} \{v_i(r_\gamma) + w(\mathbf{c}^\gamma)\}.$$
(We remark here a slight abuse of notation: the domain of $w$ is $kH$, and so $w(\mathbf{c}^\gamma)$ must be understood to mean $w(\mathbf{b}^\gamma)$, where $b_j = y_j - 1$ inside the ring $kH$ for each $1\leq j\leq e$. That is, $b_j$ is the ``obvious'' lift of $c_j$ from $kH/P$ to $kH$. This relies on the assumption that $P$ is faithful.)

Note in particular that $f_i|_{Q'} = v_i$, and $\mathrm{gr}_{f_i} Q$ is a commutative domain, again by \cite[proof of Theorem 8.6]{ardakovInv}.
\end{defn}

\begin{lem}\label{lem: F acts on valuations of Q}
$\sigma$ induces a group action of $F$ on the set of valuations of $Q$.
\end{lem}

\begin{proof}
Let $u$ be a valuation of $Q$. Again, $G$ acts on the set of valuations of $Q$ by $(g\cdot u)(x) = u(g^{-1}xg)$. Now, any $n\in H$ can be considered as an element of $Q^\times$, so that
\begin{align*}
(n\cdot u)(x) = u(n^{-1}xn) = u(n^{-1}) + u(x) + u(n) = u(x).&\qedhere
\end{align*}
\end{proof}

In the following lemma, we crucially use the fact that $\omega$ has been chosen to be $F$-stable.

\begin{lem}\label{lem: orbit of valuations}
$f_1, \dots, f_s$ is the $F$-orbit of $f_1$.
\end{lem}

\begin{proof}
Take some $g\in F$ and some $1 \leq i,j \leq s$ such that $v_j = g\cdot v_i$. We will first show that, for all $x\in Q$, we have $f_j(x) \leq g\cdot f_i(x)$. Indeed, as $f_j|_{Q'} = v_j = g\cdot v_i = g\cdot f_i|_{Q'}$, and both $f_j$ and $g\cdot f_i$ are valuations, it will suffice to show that $(w(c_k) = ) f_j(c_k) \leq g\cdot f_i(c_k)$ for each $1\leq k\leq e$.

Fix some $1\leq k\leq e$. Write $y_k^g = z\mathbf{y}^\alpha$ for some $\alpha\in \mathbb{Z}_p^e$ and $z\in Z$, so that
\begin{align*}
c_k^g = y_k^g - 1 &= z\mathbf{y}^\alpha - 1\\
&= (z-1) + z\left(\prod_{i=1}^e (1+c_i)^{\alpha_i} - 1\right)&\text{(ordered product)}\\
&= (z-1) + z\left(\sum_{\beta\neq 0} \iota\binom{\alpha}{\beta} \mathbf{c}^\beta\right),
\end{align*}
and hence
\begin{align*}
(g\cdot f_i)(c_k) &= \inf\left\{ v_i(z-1), w(\mathbf{c}^\beta) \;\middle|\; \iota\binom{\alpha}{\beta} \neq 0 \right\}
&\txt{by Definition \ref{defn: filtrations f_i}}
\\
&\geq \inf\left\{ w(z-1), w(\mathbf{c}^\beta) \;\middle|\; \iota\binom{\alpha}{\beta} \neq 0 \right\}
&\txt{by Definition \ref{defn: filtration v_1}}
\\
&= w(c_k^g),
\end{align*}
with this final equality following from \cite[Lemma 8.5(b)]{ardakovInv}. But now, as $\omega$ has been chosen to be $G$-stable, $w$ is also $G$-stable (see the remark in Definition \ref{defn: valuation w}), so that $w(c_k^g) = w(c_k)$.

Now, we have shown that, if $v_j = g\cdot v_i$ on $Q'$, then $f_j \leq g\cdot f_i$ on $Q$.

Similarly, we have $v_i = g^{-1}\cdot v_j$ on $Q'$, so $f_i \leq g^{-1}\cdot f_j$ on $Q$. But $f_i(x) \leq f_j(x^{g^{-1}})$ for all $x\in Q$ is equivalent to $f_i(y^g) \leq f_j(y)$ for all $y\in Q$ (by setting $x = y^g$). Hence we have $f_i = g\cdot f_j$ on $Q$, and we are done.
\end{proof}

As in Definition \ref{defn: filtration v}:

\begin{defn}\label{defn: filtration f}
Let $f$ be the filtration of $Q$ defined by $$\displaystyle f(x) = \inf_{1\leq i\leq s} f_i(x)$$ for each $x\in Q$.
\end{defn}

We now verify that the relationship between $f$ and $v$ is the same as that between the $f_i$ and the $v_i$ (Definition \ref{defn: filtrations f_i}).

\begin{lem}\label{lem: formula for f}
Take any $x\in Q$, and write it in standard form as
$$x = \sum_{\gamma\in \mathbb{N}^e} r_\gamma \mathbf{c}^\gamma.$$
Then we have
$$f(x) = \inf_{\gamma\in \mathbb{N}^e} \{v(r_\gamma) + w(\mathbf{c}^\gamma)\}.$$
\end{lem}

\begin{proof}
Immediate from Definitions \ref{defn: filtration v}, \ref{defn: filtrations f_i} and \ref{defn: filtration f}.
\end{proof}

Now we can extend Lemma \ref{lem: gr Q'} to $Q$:

\begin{lem}\label{lem: gr Q}
$\displaystyle \gr_f Q \cong \bigoplus_{i=1}^s \gr_{f_i} Q$.
\end{lem}

\begin{proof}
As in the proof of Lemma \ref{lem: gr Q'}, we get an injective map
$$\displaystyle \gr_f Q \to \bigoplus_{i=1}^s \gr_{f_i} Q.$$
The proof of \cite[8.6]{ardakovInv} gives a map $$(\gr_v (kZ/\mathfrak{p}))[Y_1, \dots, Y_e] \to \gr_f (kH/P)$$ and isomorphisms $$(\gr_{v_i} (kZ/\mathfrak{p}))[Y_1, \dots, Y_e] \cong \gr_{f_i} (kH/P)$$ for each $1\leq i\leq s$, in each case mapping $Y_j$ to $\gr(c_j)$ for each $1\leq j\leq e$.

Now, $\gr\, kH$ is a \emph{$\gr$-free} \cite[\S I.4.1, p. 28]{LVO} $\gr\, kZ$-module with respect to $f$ and each $f_i$, and each of these filtrations is discrete on $kH$ by construction (see \cite[Corollary 6.2 and proof of Theorem 7.3]{ardakovInv}), so by \cite[I.6.2(3)]{LVO}, $kH$ is a \emph{filt-free} $kZ$-module with respect to $f$ and each $f_i$; and by \cite[I.6.14]{LVO}, these maps extend to a map $(\gr_v Q')[Y_1, \dots, Y_e] \to \gr_f Q$ and isomorphisms $(\gr_{v_i} Q')[Y_1, \dots, Y_e] \cong \gr_{f_i} Q$ for each $i$.

Applying Lemma \ref{lem: orbit of valuations} to each $1\leq i\leq s$, we get isomorphisms
$$ (\gr_{v_i} Q')[Y_1, \dots, Y_e] \to \gr_{f_i} Q,$$
which give a commutative diagram

\centerline{
\xymatrix{
(\gr_v Q')[Y_1, \dots, Y_e] \ar[r]^-\cong \ar[d]& \displaystyle\bigoplus_{i=1}^s (\gr_{v_i} Q')[Y_1, \dots, Y_e] \ar[d]^-\cong\\
\gr_f Q \ar@{^{(}->}[r]& \displaystyle \bigoplus_{i=1}^s \gr_{f_i} Q.
}
}

Hence clearly all maps in this diagram are isomorphisms.
\end{proof}

Now we return to the ring $R = Q*F$ defined in (\ref{eqn: Q*F}).

\begin{defn}
We can extend the filtration $f$ on $Q$ to an $F$-stable filtration on $R$ by giving elements of the basis $\overline{F}$ value $0$. That is, writing $\overline{F} = \{\overline{g}_1, \dots, \overline{g}_m\}$, any element of $Q*F$ can be expressed uniquely as $\sum_{r=1}^m \overline{g}_r x_r$ for some $x_r \in Q$: the assignment
\begin{align*}
Q*F &\to \mathbb{R}\cup \{\infty\} \\
\sum_{r=1}^m \overline{g}_r x_r &\mapsto \inf_{1\leq r\leq m} \Big\{f(x_r)\Big\}
\end{align*}
is clearly a filtration on $Q*F$ whose restriction to $Q$ is just $f$. We will temporarily refer to this filtration as $\hat{f}$, though later we will drop the hat and simply call it $f$.
\end{defn}

Note that, for any real number $\lambda$,
\begin{align*}
(Q*F)_{\hat{f}, \lambda} &= \bigoplus_{i=1}^m \overline{g_i} (Q_{f,\lambda}),\\
(Q*F)_{\hat{f}, \lambda^+} &= \bigoplus_{i=1}^m \overline{g_i} (Q_{f,\lambda^+}),\\
\end{align*}
so that
\begin{align*}
\gr_{\hat{f}} (Q*F) &= \bigoplus_{\lambda\in \mathbb{R}} \left(\bigoplus_{i=1}^m \overline{g_i} (Q_{f,\lambda} / Q_{f,\lambda^+})\right)\\
&=  \bigoplus_{i=1}^m \overline{g_i}\left( \bigoplus_{\lambda\in \mathbb{R}}(Q_{f,\lambda} / Q_{f,\lambda^+})\right)
=  \bigoplus_{i=1}^m \overline{g_i}\left(\gr_{f}(Q)\right).
\end{align*}

That is, given the data of a crossed product $Q*F$ as in (\ref{eqn: Q*F}), we may view $\gr_{\hat{f}} (Q*F)$ as $\gr_f(Q) * F$ in a natural way.

We will finally record this as:

\begin{lem}\label{lem: isomorphisms of grs}
\begin{align*}
\displaystyle \gr_f(Q*F) = \gr_f(Q) * F &\cong \left(\bigoplus_{i=1}^s \gr_{f_i} Q\right) * F \\
&\cong \left( \bigoplus_{i=1}^s (\gr_{v_i} Q')[Y_1, \dots, Y_e] \right) * F,
\end{align*}
where each $\gr_{f_i} Q$ (or equivalently each $\gr_{v_i} Q'$) is a domain (see Definition \ref{defn: filtrations f_i}). $F$ permutes the $f_i$ (or equivalently the $v_i$) transitively by conjugation.\qed
\end{lem}

\textit{Proof of Theorem C.} This is Lemma \ref{lem: isomorphisms of grs}.\qed

\subsection{Automorphisms trivial on a free abelian quotient}\label{subsection: automs trivial on N/L}

\textbf{We will fix the following notation for this subsection.}

\begin{notn}
Let $H$ be a nilpotent but non-abelian $p$-valuable group with centre $Z$. Write $H'$ for its isolated derived subgroup \cite[Theorem B]{woods-struct-of-G}. Suppose we are given a closed isolated proper characteristic subgroup $L$ of $H$ which contains $H'$ and $Z$. (We will show that such an $L$ always exists in Lemma \ref{lem: H'Z is not all of H}.) Fix a $p$-valuation $\omega$ on $H$ satisfying (\omegaprop{L}) (which is possible by Corollary \ref{cor: there exists an F-stable omega satisfying omegaprop}).

Let $\{g_{m+1}, \dots, g_n\}$ be an ordered basis for $Z$. Using Lemma \ref{lem: extending ordered bases} twice, extend this to an ordered basis $\{g_{l+1}, \dots, g_n\}$ for $L$, and then extend this to an ordered basis $\{g_1, \dots, g_n\}$ for $H$. Diagrammatically:
\begin{align*}
B_H = \big\{ \underbrace{g_1, \dots, g_l}_{B_{H/L}}, \;\; \underbrace{g_{l+1}, \dots, g_m}_{B_{L/Z}}, \;\; \underbrace{g_{m+1}, \dots, g_n}_{B_Z} \big\},
\end{align*}
extending the notation of the remark after Lemma \ref{lem: extending ordered bases} in the obvious way. Here, $0 < l \leq m < n$, corresponding to the chain of subgroups $1 \lneq Z \leq L \lneq H$.

Let $k$ be a field of characteristic $p$. As before, let $\mathfrak{p}$ be a faithful prime ideal of $kZ$, so that $P := \mathfrak{p} kH$ is a faithful prime ideal of $kH$, and write $b_j = g_j - 1 \in kH/P$ for all $1\leq j\leq m$.

In this subsection, we will write:
\begin{itemize}
\item for each $\alpha\in \mathbb{N}^m$, $\mathbf{b}^\alpha$ means the (ordered) product $b_1^{\alpha_1} \dots b_m^{\alpha_m}\in kH/P$,
\item for each $\alpha\in \mathbb{Z}_p^m$, $\mathbf{g}^\alpha$ means the (ordered) product $g_1^{\alpha_1} \dots g_m^{\alpha_m}\in H$,
\item for each $\alpha\in \mathbb{N}^m$, $\langle \alpha, \omega(\mathbf{g}) \rangle$ means $\displaystyle \sum_{i=1}^m \alpha_i \omega(g_i)$.
\end{itemize}
Note the use of $m$ rather than $n$ in each case. This means that every element $x\in H$ may be written uniquely as
$$x = z\mathbf{g}^\alpha$$
for some $\alpha\in \mathbb{Z}_p^m$ and $z\in Z$; and every element $y\in kH/P$ may be written uniquely as
$$y = \sum_{\gamma\in \mathbb{N}^m} r_\gamma \mathbf{b}^\gamma$$
for some elements $r_\gamma\in kZ/\mathfrak{p}$.
\end{notn}

Recall the definitions of the filtrations $w$ on $kH$ (Definition \ref{defn: valuation w}), $v$ on $kZ/\mathfrak{p}$ (Definition \ref{defn: filtration v}) and $f$ on $kH/P$ (Definition \ref{defn: filtration f}). We will continue to abuse notation slightly for $w$, as in Definition \ref{defn: filtrations f_i}.

Recall also that, as we have chosen $\omega$ to satisfy (\omegaprop{L}), we have that $$w(b_1) = \dots = w(b_l) < w(b_r)$$ for all $r>l$.

Let $\sigma$ be an automorphism of $H$, and suppose that, when naturally extended to an automorphism of $kH$, it satisfies $\sigma(P) = P$. Hence we will consider $\sigma$ as an automorphism of $kH/P$, preserving the subgroup $H\subseteq (kH/P)^\times$.

\begin{cor}\label{cor: f implies w}
With the above notation, fix $1\leq i\leq l$. If $f(\sigma(b_i) - b_i) > f(b_i)$, then $w(\sigma(b_i) - b_i) > w(b_i)$.
\end{cor}

\begin{proof}
Write in standard form
\begin{align}
\sigma(b_i) - b_i = \sum_{\gamma\in \mathbb{N}^m} r_\gamma \mathbf{b}^\gamma,\nonumber
\end{align}
for some $r_\gamma\in kZ$, and suppose that $f(\sigma(b_i) - b_i) > f(b_i)$. That is, by Lemma \ref{lem: formula for f},
$$v(r_\gamma) + w(\mathbf{b}^\gamma) > w(b_i)$$
for each fixed $\gamma\in \mathbb{N}^m$.

We will show that $w(r_\gamma) + w(\mathbf{b}^\gamma) > w(b_i)$ for each $\gamma$. We deal with two cases.

\textbf{Case 1:} $w(\mathbf{b}^\gamma) > w(b_i)$. Then, as $w$ takes non-negative values on $kH$, we are already done.

\textbf{Case 2:} $w(\mathbf{b}^\gamma) \leq w(b_i)$. Then, by (\omegaprop{L}), we have either $w(r_\gamma) > w(b_i)$ or $w(r_\gamma) = 0$. In the former case, we are done automatically, so assume we are in the latter case and $w(r_\gamma) = 0$. Then, by \cite[6.2]{ardakovInv}, $r_\gamma$ must be a unit in $kZ$, and so $f(r_\gamma) = 0$ by Lemma \ref{lem: basic filtration property}, a contradiction.

Hence $w(r_\gamma) + w(\mathbf{b}^\gamma) > w(b_i)$ for all $\gamma\in \mathbb{N}^m$. But, as $w$ is discrete by \cite[6.2]{ardakovInv}, we may now take the infimum over all $\gamma\in \mathbb{N}^m$, and the inequality remains strict.
\end{proof}

Let $\sigma$ be an automorphism of $H$, and recall that $H/L$ is a free abelian pro-$p$ group of rank $l$. Choose a basis $e_1, \dots, e_l$ for $\mathbb{Z}_p^l$; then the map $g_iL \mapsto e_i$ for $1\leq i\leq l$ is an isomorphism $j: H/L \to \mathbb{Z}_p^l$. As $L$ is characteristic in $H$ by assumption, $\sigma$ induces an automorphism of $H/L$, which gives a matrix $M_\sigma\in GL_l(\mathbb{Z}_p)$ under this isomorphism.

Write $$\overline{\omega}: H/L\to \Rinfty$$ for the quotient $p$-valuation on $H/L$ induced by $\omega$, i.e. $$\overline{\omega}(xL) = \sup_{\ell\in L} \{\omega(x\ell)\}$$ -- note that this is just the $(t,p)$-filtration (Definition \ref{defn: (t,p)-filtration}), as we have chosen $\omega$ to satisfy (\omegaprop{L}). Then write $$\Omega: \mathbb{Z}_p^l\to \Rinfty$$ for the map $\Omega = \overline{\omega}\circ j^{-1}$, the $(t,p)$-filtration on $\mathbb{Z}_p^l$ corresponding to $\overline{\omega}$ under the isomorphism $j$.

\begin{rk}
If $x\in \mathbb{Z}_p^l$ has $\Omega(x) \geq t+1$, then $x\in p\mathbb{Z}_p^l$, by the definition of the $(t,p)$-filtration.
\end{rk}

We will write $\Gamma(1) = 1+pGL_l(\mathbb{Z}_p)$ for the \emph{first congruence subgroup} of $GL_l(\mathbb{Z}_p)$, the open subgroup of $GL_l(\mathbb{Z}_p)$ whose elements are congruent to the identity element modulo $p$.

\begin{cor}\label{cor: if sigma increases f part 1}
With the above notation, if $f(\sigma(b_i) - b_i) > f(b_i)$ for all $1\leq i\leq l$, then $M_\sigma\in\Gamma(1)$.
\end{cor}

\begin{proof}
We have, for all $1\leq i\leq l$,
\begin{align*}
f(\sigma(b_i) - b_i) > f(b_i) &\implies w(\sigma(b_i) - b_i) > w(b_i) &\txt{by Corollary \ref{cor: f implies w}},\\
&\implies \omega(\sigma(g_i)g_i^{-1}) > \omega(g_i) &\txt{by Theorem \ref{thm: w(x-1) = omega(x) attempt 2}},
\end{align*}
-- which is condition (\ref{eqn: sigma-1 increases omega}). Now we may invoke Theorem \ref{thm: M is in Gamma_1}.
\end{proof}

\begin{cor}\label{cor: if sigma increases f, then sigma is in Gamma_1}
Suppose now further that $\sigma$ is an automorphism of $H$ of \emph{finite order}. If $p > 2$ and $f(\sigma(b_i) - b_i) > f(b_i)$ for all $1\leq i\leq l$, then $\sigma$ induces the identity automorphism on $H/L$.\qed
\end{cor}

\begin{proof}
We have shown that $M_{\sigma} \in \Gamma(1)$, which is a $p$-valuable (hence torsion-free) group for $p > 2$ by \cite[Theorem 5.2]{DDMS}; and if $\sigma$ has finite order, then $M_\sigma$ must have finite order. So $M_\sigma$ is the identity map.
\end{proof}

\textit{Proof of Theorem D.} This now follows from Corollaries \ref{cor: if sigma increases f part 1} and \ref{cor: if sigma increases f, then sigma is in Gamma_1}.\qed

\begin{rk}
When $p = 2$, $\Gamma(1)$ is no longer $p$-valuable.
\end{rk}

\begin{ex}
Let $p = 2$, and let $$H = \overline{\langle x, y, z \;|\; [x,y] = z, [x,z]=1, [y,z]=1\rangle}$$ be the ($2$-valuable) $\mathbb{Z}_2$-Heisenberg group. Let $\sigma$ be the automorphism sending $x$ to $x^{-1}$, $y$ to $y^{-1}$ and $z$ to $z$. Take $L = \overline{\langle z \rangle}$, and $P = 0$.

Write $X = x-1\in kH/P$, and likewise $Y = y-1$ and $Z = z-1$. Now,
$$\sigma(X) = \sigma(x) - 1 = x^{-1} - 1 = (1+X)^{-1} - 1 = -X + X^2 - X^3 + \dots,$$
and so $\sigma(X) - X = X^2 - X^3 + \dots$ (as $\mathrm{char}\; k = 2$). Hence $f(\sigma(X) - X) = f(X^2) > f(X)$; but
$$M_\sigma = \begin{pmatrix} -1&0 \\ 0&-1 \end{pmatrix} \in \Gamma(1, GL_2(\mathbb{Z}_2)),$$
and in particular $M_\sigma \neq 1$.
\end{ex}

\section{Extending prime ideals from $\fn(G)$}\label{section: extending primes}

\subsection{X-inner automorphisms}

\begin{defn}\label{defn: X-inner stuff}
We recall the notation of \cite[\S 5]{woods-prime-quotients}: given $R$ a ring, $G$ a group and a fixed crossed product $S$ of $R$ by $G$, we will sometimes write the structure explicitly as
$$S = R\cp{\sigma}{\tau} G,$$
where $\sigma: G\to \Aut(R)$ is the \emph{action} and $\tau: G\times G\to R^\times$ the \emph{twisting}.

Furthermore, we say that an automorphism $\varphi\in \Aut(R)$ is \emph{X-inner} if there exist nonzero elements $a,b,c,d\in R$ such that, for all $x\in R$,
$$axb = cx^\varphi d,$$
where $x^\varphi$ denotes the image of $x$ under $\varphi$. Write $\mathrm{Xinn}(R)$ for the subgroup of $\Aut(R)$ consisting of X-inner automorphisms; and, given a crossed product as in the previous paragraph, we will write $\Xinn{S}{R}{G} = \sigma^{-1}(\sigma(G) \cap \mathrm{Xinn}(R))$.
\end{defn}

\begin{lem}\label{lem: X-inner stuff}
$R$ a prime ring and $R*G$ a crossed product. Let $G_\inn := \Xinn{R*G}{R}{G}$.
\begin{itemize}
\item[(i)] If $\sigma\in\Aut(R)$ is X-inner, then $\sigma$ is trivial on the centre of $R$.
\item[(ii)] If $H$ is a subgroup of $G$ containing $G_\inn$, and $R*H$ is a prime ring, then $R*G$ is a prime ring.
\end{itemize}
\end{lem}

\begin{proof}
$ $

\begin{itemize}
\item[(i)] This follows from the description of X-inner automorphisms of $R$ as restrictions of inner automorphisms of the Martindale symmetric ring of quotients $\mathrm{Q}_{s}(R)$, and the fact that $Z(R)$ stays central in $\mathrm{Q}_{s}(R)$: see \cite[\S 12]{passmanICP} for details.
\item[(ii)] This follows from \cite[Corollary 12.6]{passmanICP}: if $I$ is a nonzero ideal of $R*G$, then $I\cap R*G_\inn$ is nonzero, and hence $I\cap R*H$ is nonzero.\qedhere
\end{itemize}
\end{proof}

\subsection{Properties of $\fn(G)$}

We prove here some important facts about the group $\fn(G)$ (defined in \cite[Theorem C]{woods-struct-of-G}).

\begin{lem}\label{lem: subgroup of G acting trivially on H/H' is FN_p}
Let $G$ be a nilpotent-by-finite compact $p$-adic analytic group with $\Delta^+ = 1$. Let $H = \fn(G)$, and write $$K := K_G(H) = \{x\in G \;|\; [H,x]\leq H'\},$$ where $H'$ denotes the isolated derived subgroup of $H$ \cite[Theorem B]{woods-struct-of-G}. Then $K = H$.
\end{lem}

\begin{proof}
Firstly, note that $K$ clearly contains $H$, by definition of $H'$.

Secondly, suppose that $H$ is $p$-saturable. By the same argument as in \cite[Lemma 4.3]{woods-struct-of-G}, $K$ acts nilpotently on $H$, and so $K$ acts nilpotently on the Lie algebra $\mathfrak{h}$ associated to $H$ under Lazard's isomorphism of categories \cite{lazard}. That is, we get a group representation $\mathrm{Ad}: K\to \Aut(\mathfrak{h})$, and $(\mathrm{Ad}(k) - 1)(\mathfrak{h}_i) \subseteq \mathfrak{h}_{i+1}$ for each $k\in K$ and each $i$. (Here, $\mathfrak{h}_i$ denotes the $i$th term in the lower central series for $\mathfrak{h}$.)

Choosing a basis for $\mathfrak{h}$ adapted to the flag
$$\mathfrak{h} \supsetneq \mathfrak{h}_2 \supsetneq \dots \supsetneq \mathfrak{h}_r = 0,$$
we see that $\mathrm{Ad}$ is a representation of $K$ for which $\mathrm{Ad}(k) - 1$ is strictly upper triangular for each $k\in K$; in other words, $\mathrm{Ad}: K\to \mathcal{U}$, where $\mathcal{U}$ is a closed subgroup of some $GL_n(\mathbb{Z}_p)$ consisting of unipotent upper-triangular matrices. Hence the image $\mathrm{Ad}(K)$ is nilpotent and torsion-free.

Furthermore, $\ker \mathrm{Ad}$ is the subgroup of $K$ consisting of those elements $k$ which centralise $\mathfrak{h}$, and therefore centralise $H$. This clearly contains $Z(H)$. On the other hand, if $k$ centralises $H$, then $k$ is centralised by $H$, an open subgroup of $G$, and so $k$ must be contained in $\Delta$. But $\Delta = Z(H)$ by \cite[Lemma 5.1(ii)]{woods-struct-of-G}

Hence $K$ is a central extension of two nilpotent, torsion-free compact $p$-adic analytic groups of finite rank, and so is such a group itself; hence $K$ is nilpotent $p$-valuable by \cite[Lemma 2.3]{woods-struct-of-G}, and so must be contained in $H$ by definition of $\fn(G)$.

Now suppose $H$ is not $p$-saturable, and fix a $p$-valuation on $H$. Conjugation by $k\in K$ induces the trivial automorphism on $H/H'$, so by \cite{lazard} it does also on $\Sat(H/H')$, which is naturally isomorphic to $\Sat H / (\Sat H)'$ by \cite[Lemma 3.2]{woods-struct-of-G}. This shows that $K \subseteq K_G(\Sat H)$. But now, writing $\mathfrak{h}$ for the Lie algebra associated to $\Sat H$, the same argument as above, \emph{mutatis mutandis}, shows that $K_G(\Sat H) = H$.
\end{proof}

Some properties.

\begin{lem}\label{lem: FN_p is preserved under central quotients}
Let $G$ be a compact $p$-adic analytic group with $\Delta^+ = 1$, and write $H = \fn(G)$. If $H$ is not abelian, then $H/Z = \fn(G/Z)$.
\end{lem}

\begin{proof}
$H/Z$ is a nilpotent $p$-valuable open normal subgroup of $G/Z$, so must be contained within $\fn(G/Z)$. Conversely, the preimage in $G$ of $\fn(G/Z)$ is a central extension of $Z$ by $\fn(G/Z)$, two nilpotent and torsion-free groups, and hence is nilpotent and torsion-free, so must be $p$-valuable by \cite[Lemma 2.3]{woods-struct-of-G}, which shows that it must be contained within $H$.
\end{proof}

Recall that, if $J$ is a closed isolated subgroup of $H$, then there exists a unique smallest isolated orbital subgroup of $H$ containing $J$, which we call its \emph{isolator}, and denote $\mathrm{i}_H(J)$, as in \cite[Definition 1.6]{woods-struct-of-G}.

The (closed, isolated orbital, characteristic) subgroup $\mathrm{i}_H(H'Z)$ of $H = \fn(G)$ will be crucial throughout this section, so we record some results.

\begin{lem}\label{lem: H'Z is not all of H}
Let $H$ be a nilpotent $p$-valuable group. If $H$ is not abelian, then $H \neq \mathrm{i}_H(H'Z)$.
\end{lem}

\begin{proof}
Suppose first that $H$ is $p$-saturated, and write $\mathfrak{h}$ and $\mathfrak{z}$ for the Lie algebras of $H$ and $Z$ respectively under Lazard's correspondence \cite{lazard}. If $\mathfrak{h} = \mathfrak{h}_2\mathfrak{z}$ (writing $\mathfrak{h}_2$ for the second term in the lower central series of $\mathfrak{h}$), then by applying $[\mathfrak{h},-]$ to both sides, we see that $\mathfrak{h}_2 = \mathfrak{h}_3$. But as $\mathfrak{h}$ is nilpotent, this implies that $\mathfrak{h}_2 = 0$, so that $\mathfrak{h}$ is abelian, a contradiction.

When $H$ is not $p$-saturated: note that $\mathrm{i}_H(H'Z) = \Sat\!(H'Z)\cap H$, by \cite[Lemma 3.1]{woods-struct-of-G}, and so that $\Sat\!(H/\mathrm{i}_H(H'Z)) \cong \Sat\!(H)/\Sat\!(H'Z)$ by Lemma \cite[Lemma 3.2]{woods-struct-of-G}. Hence $H/\mathrm{i}_H(H'Z)$ has the same (in particular non-zero) rank as $\Sat\!(H)/\Sat\!(H'Z)$.
\end{proof}

\begin{lem}\label{lem: subgroup of G acting trivially on H/H'Z is FN_p}
Let $G$ be a nilpotent-by-finite compact $p$-adic analytic group with $\Delta^+ = 1$. Let $H = \fn(G)$, and assume that $H$ is not abelian. Write $$M := M_G(H) = \{x\in G \;|\; [H,x]\leq \mathrm{i}_H(H'Z)\},$$ where $H'$ denotes the isolated derived subgroup of $H$, and $Z$ the centre of $H$. Then $M = H$.
\end{lem}

\begin{proof}
Clearly $Z\leq M$. We will calculate $M/Z$.

First, note that $\mathrm{i}_H(H'Z)/Z$ is an isolated normal subgroup of $H/Z$, as the quotient is isomorphic to $H/\mathrm{i}_H(H'Z)$, which is torsion-free. Also, as $\mathrm{i}_H(H'Z)$ contains $H'Z$ and hence $\overline{[H,H]}Z$ as an open subgroup, clearly $\mathrm{i}_H(H'Z)/Z$ contains $\overline{[H,H]}Z/Z$ as an open subgroup, so that $\mathrm{i}_H(H'Z)/Z \leq \mathrm{i}_{H/Z}(\overline{[H,H]}Z/Z)$.

Now, $[H/Z, H/Z] = [H,H]Z/Z$ as abstract groups, so by taking their closures followed by their $(H/Z)$-isolators, we see that
$$(H/Z)' = \mathrm{i}_{H/Z}\big(\overline{[H,H]Z/Z}\big) = \mathrm{i}_{H/Z}\big(\overline{[H,H]}Z/Z\big),$$
so that 
$$\mathrm{i}_H(H'Z)/Z = (H/Z)'.$$
But $x\in M$ if and only if $[H,x]\leq \mathrm{i}_H(H'Z)$, which is equivalent to $[H/Z, xZ]\leq (H/Z)'$, or in other words $xZ\in K_{G/Z}(H/Z) = H/Z$ by Lemma \ref{lem: subgroup of G acting trivially on H/H' is FN_p}. So $M/Z = H/Z$, and hence $M = H$.
\end{proof}

\subsection{The extension theorem}

\begin{propn}\label{propn: main prime extension theorem, Delta+ = 1}
Fix a prime $p>2$ and a finite field $k$ of characteristic $p$. Let $G$ be a nilpotent-by-finite compact $p$-adic analytic group with $\Delta^+ = 1$. Suppose $H = \fn(G)$, and write $F = G/H$. Let $P$ be a $G$-stable, faithful prime ideal of $kH$. Let $(kG)_\alpha$ be a central 2-cocycle twist of $kG$ with respect to a standard (Definition \ref{defn: recall standard cp}) decomposition $$kG = kH\cp{\sigma}{\tau} F,$$ for some $\alpha\in Z^2_\sigma(F, Z((kH)^\times))$, as in \cite[Theorem 5.12]{woods-prime-quotients}. Then $P(kG)_\alpha$ is a prime ideal of $(kG)_\alpha$.
\end{propn}

\begin{proof}
First, we note that the claim that $P(kG)_\alpha$ is a prime ideal of $(kG)_\alpha$ is equivalent to the claim that $$(kG)_\alpha/P(kG)_\alpha = kH/P\cp{\sigma}{\tau\alpha} F$$ is a prime ring.

\textbf{Case 1.} Suppose that $G$ centralises $Z$.

If $H$ is abelian, so that $H = Z$, then every $g\in G$ is centralised by $Z$, an open subgroup of $G$. Hence $g\in \Delta$, i.e. $G = \Delta$. But, by \cite[Theorem D]{woods-struct-of-G}, $\Delta\leq H$, and so we have $G = H$ and there is nothing to prove.

So suppose henceforth that $Z\lneq H$, and write $L := \mathrm{i}_H(H'Z)$, so that, by Lemma \ref{lem: H'Z is not all of H}, we have $L \lneq H$. As the decomposition of $kG$ is standard, we may view $F$ as a subset of $G$.

The idea behind the proof is as follows. We will construct a crossed product $R*F'$, where $R$ is a certain commutative domain and $F'$ is a certain subgroup of $F$, with the following property: if $R*F'$ is a prime ring, then $P(kG)_\alpha$ is a prime ideal. Then, by using the well-understood structure of $R$, we will show that the action of $F'$ on $R$ is X-outer (in the sense of Definition \ref{defn: X-inner stuff}), so that $R*F'$ is a prime ring.

By Corollary \ref{cor: there exists an F-stable omega satisfying omegaprop}, we can see that $H$ admits an $F$-stable $p$-valuation $\omega$ satisfying (\omegaprop{L}). Hence, in the notation of \S \ref{subsection: ring filtrations, w, f}, we may define the filtration $w$ from $\omega$ as in Definition \ref{defn: valuation w}. Furthermore, we write $$Q' = \mathbf{Q}(kZ/P\cap kZ),\qquad Q = Q'\tensor{kZ} kN,$$ as in \S \ref{subsection: Q', Q, filtrations}; and we endow $Q$ with the $F$-orbit of filtrations $f_i$ ($1\leq i\leq s$) and the filtration $f$ of Definitions \ref{defn: filtrations f_i} and \ref{defn: filtration f}, defined in terms of the filtration $w$ above.

By \cite[2.1.16(vii)]{MR}, in order to show that the crossed product
\begin{align}\label{eqn: kH/P*F}
kH/P\cp{\sigma}{\tau\alpha} F
\end{align}
is a prime ring, it suffices to show that the related crossed product
\begin{align}\label{eqn: Q*F}
Q\cp{\sigma}{\tau\alpha}F
\end{align}
is prime, where this crossed product is defined in \S \ref{subsection: Q', Q, filtrations}. Then, by \cite[II.3.2.7]{LVO}, it suffices to show that
\begin{align}\label{eqn: gr_f(Q*F)}
\gr_f(Q*F)
\end{align}
is prime. Details of this graded ring are given in Lemma \ref{lem: isomorphisms of grs}: in particular, note that
$$\gr_f(Q*F) \cong \left(\bigoplus_{i=1}^s \gr_{f_i} Q\right)*F.$$

Now, as noted in Definition \ref{defn: filtrations f_i}, each $\gr_{f_i} Q$ is a commutative domain, and by construction, $F$ permutes the summands $\gr_{f_i} Q$ transitively. So by \cite[Corollary 14.8]{passmanICP} it suffices to show that
\begin{align}\label{eqn: gr_(f_1)(Q)*F'}
\gr_{f_1} Q * F'
\end{align}
is prime, where $F' = \mathrm{Stab}_F(f_1)$.

\begin{notn}
We set up notation in order to be able to apply the results of \S \ref{subsection: automs trivial on N/L}. Let $\{y_{m+1}, \dots, y_n\}$ be an ordered basis for $Z$, which we extend to an ordered basis $\{y_{l+1}, \dots, y_n\}$ for $L$, which we extend to an ordered basis $\{y_1, \dots, y_n\}$ for $H$. Set $b_i = y_i - 1\in kH/P$, and let $Y_i = \gr_{f_1}(b_i)$ for all $1\leq i\leq m$. Then
$$\gr_{f_1} Q \cong \left( \gr_{v_1} Q' \right) [Y_1, \dots, Y_m].$$
\end{notn}

The ring on the right-hand side inherits a crossed product structure
\begin{align}\label{eqn: gr(Q')[Y]*F}
\left( \gr_{v_1} Q' \right) [Y_1, \dots, Y_m] * F'.
\end{align}
from (\ref{eqn: gr_(f_1)(Q)*F'}). Writing $R := \left( \gr_{v_1} Q' \right) [Y_1, \dots, Y_m]$, we have now shown, by passing along the chain
\begin{center}
(\ref{eqn: gr(Q')[Y]*F}) $\rightarrow$ (\ref{eqn: gr_(f_1)(Q)*F'}) $\rightarrow$ (\ref{eqn: gr_f(Q*F)}) $\rightarrow$ (\ref{eqn: Q*F}) $\rightarrow$ (\ref{eqn: kH/P*F}),
\end{center}
that we need only show that $R*F'$ is prime.

Write $F'_\inn$ for the subgroup of $F'$ acting on $R$ by X-inner automorphisms in the crossed product (\ref{eqn: gr(Q')[Y]*F}), i.e.
$$F'_\inn = \Xinn{R*F'}{R}{F'}$$
in the notation of Definition \ref{defn: X-inner stuff}. By the obvious abuse of notation, we will denote this action as the map of sets $\gr\;\sigma: F'\to \Aut(R)$.

Take some $g\in F'$. If $\gr\;\sigma(g)$ acts non-trivially on $R$, then as $R$ is commutative, we have $g\not\in F'_\inn$. Hence, as by Lemma \ref{lem: X-inner stuff}(ii) we need only show that $R* F'_\inn$ is prime, we may restrict our attention to those $g\in F'$ that act trivially on $R$. In particular, such a $g\in F'$ must centralise each $Y_i$. But
$$\gr\;\sigma(g)(Y_i) = Y_i \Leftrightarrow f(\sigma(g)(b_i) - b_i) > f(b_i).$$
Now we see (as $p>2$) from Corollary \ref{cor: if sigma increases f, then sigma is in Gamma_1} that $\sigma(g)$ induces the identity automorphism on $H/L$, and hence from Lemma \ref{lem: subgroup of G acting trivially on H/H'Z is FN_p} that $g\in H$. That is, $F'_\inn$ is the trivial group, so that $R*F'_\inn = R$ is automatically prime.

\textbf{Case 2.} Suppose some $x\in F$ does not centralise $Z$. Write $F_\inn$ for the subgroup of $F$ acting by X-inner automorphisms on $kH/P$ in the crossed product (\ref{eqn: kH/P*F}), i.e.
$$F_\inn := \Xinn{(kG)_\alpha / P(kG)_\alpha}{kH/P}{F}.$$
Then, by Lemma \ref{lem: X-inner stuff}(i), $x\not\in F_\inn$; so $F_\inn$ is contained in $\mathbf{C}_F(Z)$, and we need only prove that the sub-crossed product $(kH/P)*\mathbf{C}_F(Z)$ is prime by Lemma \ref{lem: X-inner stuff}(ii). This reduces the problem to Case 1.
\end{proof}

\begin{propn}\label{propn: main prime extension theorem, Delta+ not equal to 1}
Let $G$ be a nilpotent-by-finite compact $p$-adic analytic group, and $k$ a finite field of characteristic $p>2$. Let $H = \fn(G)$, and write $F = G/H$. Let $P$ be a $G$-stable, almost faithful prime ideal of $kH$.
Then $PkG$ is prime.
\end{propn}

\begin{proof}
We assume familiarity with \cite[Lemma 1.1]{woods-prime-quotients}, and adopt the notation of \cite[Notation 1.2]{woods-prime-quotients} for this proof.

Let $e\in \cpi(P)$, and write $f_H = \cpisum{e}{H}$, $f = \cpisum{e}{G}$. Then $PkG$ is a prime ideal of $kG$ if and only if $f\cdot\overline{PkG}$ is prime in $f\cdot\overline{kG}$.

Write $H_1 = \mathrm{Stab}_H(e)$ and $G_1 = \mathrm{Stab}_G(e)$. Then, by the Matrix Units Lemma \cite[Lemma 6.1]{woods-prime-quotients}, we get an isomorphism
$$f\cdot\overline{kG}\cong M_s(e\cdot\overline{kG_1})$$
for some $s$, under which the ideal $f\cdot\overline{PkG}$ is mapped to $M_s(e\cdot\overline{P_1kG_1})$, where $P_1$ is the preimage in $kH_1$ of $e\cdot\overline{P}\cdot e$. It is easy to see that $P_1$ is prime in $kH_1$; indeed, applying the Matrix Units Lemma to $kH$, we get
$$f_H\cdot\overline{kH} \cong M_{s'} (e\cdot\overline{kH_1}),$$
under which $f_H\cdot\overline{P} \mapsto M_{s'} (e\cdot\overline{P_1})$, so that $P_1$ is prime by Morita equivalence (see e.g. \cite[Lemma 1.7]{woods-prime-quotients}). We also know from \cite[Lemma 6.6]{woods-prime-quotients} (or the remark after\cite[Lemma 6.2]{woods-prime-quotients}) that
$$P^\dagger = \bigcap_{h\in H} (P_1^\dagger)^h.$$
Now, writing $q$ to denote the natural map $G \to G/\Delta^+$,
$$q\left(\left(P_1^\dagger \cap \Delta\right)^{h}\right) = q\left(P_1^\dagger \cap \Delta\right)$$
for all $h\in H$, as $q(\Delta) = Z(q(H))$ by definition of $H$ (see \cite[Lemma 5.1(ii)]{woods-struct-of-G}); and so
$$q\left(P^\dagger \cap \Delta\right) = q\left(P_1^\dagger \cap \Delta\right) = q(1).$$
But $q\left(P_1^\dagger\right)$ is a normal subgroup of the nilpotent group $q\left(H_1\right)$. Hence, as the intersection of $q\left(P_1^\dagger\right)$ with the centre $q(\Delta)$ of $q(H)$ is trivial, we must have that $q\left(P_1^\dagger\right)$ is trivial also \cite[5.2.1]{rob}. That is, $P_1^\dagger \leq \Delta^+(H_1) = \Delta^+$.

Now, in order to show that $M_s(e\cdot\overline{P_1 kG_1})$ is prime, we may equivalently (by Morita equivalence) show that $e\cdot\overline{P_1 kG_1}$ is prime. By \cite[Theorem 5.12]{woods-prime-quotients}, we get an isomorphism
$$e\cdot\overline{kG_1} \cong M_t\big((k'[[G_1/\Delta^+]])_\alpha\big),$$
for some integer $t$, some finite field extension $k'/k$, and a central 2-cocycle twist (see above or \cite[Definition 5.4]{woods-prime-quotients}) of $k'[[G_1/\Delta^+]]$ with respect to a standard crossed product decomposition
$$k'[[G_1/\Delta^+]] = k'[[H_1/\Delta^+]]\cp{\sigma}{\tau} (G_1/H_1)$$
given by some
$$\alpha\in Z^2_\sigma\left(G_1/H_1, Z\left(\left(k'[[H_1/\Delta^+]]\right)^\times\right)\right).$$ Writing the image of $e\cdot\overline{P_1}$ as $M_t(\mathfrak{p})$ for some ideal $\mathfrak{p}\in k'[[H_1/\Delta^+]]$, we see by (see above or \cite[Theorem C]{woods-prime-quotients}) that $\mathfrak{p}$ is a faithful, $(G_1/\Delta^+)$-stable prime ideal of $k'[[H_1/\Delta^+]]$. It now remains only to show that the extension of $\mathfrak{p}$ to $k'[[G_1/\Delta^+]]$ is prime; but this now follows from Proposition \ref{propn: main prime extension theorem, Delta+ = 1}.
\end{proof}

\textit{Proof of Theorem A.} This follows from Proposition \ref{propn: main prime extension theorem, Delta+ not equal to 1}.\qed

\bibliography{biblio}
\bibliographystyle{plain}

\end{document}